\documentclass{article}
\usepackage{graphicx} 
\usepackage{geometry, amsthm, amsmath, amssymb, amscd, color, mathrsfs,  enumerate, verbatim}

\usepackage{hyperref}

\usepackage{courier}

\usepackage[all]{xy}

\newtheorem{Prop}[equation]{Proposition}
\newtheorem{Thm}[equation]{Theorem}
\newtheorem{Lem}[equation]{Lemma}
\newtheorem{Cor}[equation]{Corollary}
\theoremstyle{definition}\newtheorem{Def}[equation]{Definition}
\newtheorem{Ex}[equation]{Example}
\newtheorem{Rem}[equation]{Remark}

\theoremstyle{definition}
\theoremstyle{definition}\newtheorem{Conj}[equation]{Conjecture}
\theoremstyle{definition}\newtheorem{Not}[equation]{Notation}

\newcommand{\N}{\mathbb{N}}
\newcommand{\Q}{\mathbb{Q}}
\newcommand{\F}{\mathbb{F}}
\newcommand{\oQ}{\overline{\Q}}
\newcommand{\oZ}{\overline{\Z}}
\newcommand{\oP}{\overline{P}}

\newcommand{\Z}{\mathbb{Z}}
\newcommand{\R}{\mathbb{R}}
\newcommand{\oZp}{\overline{\Z_p}}
\newcommand{\oQZp}{\overline{\Z_{(p)}}}
\newcommand{\oQp}{\overline{\Q_p}} 
\newcommand{\oMp}{\overline{M_p}}

\newcommand{\Int}{\textnormal{Int}}   
\newcommand{\IntQ}{\Int_{\Q}}   
\newcommand{\An}{\mathcal{A}_n}
\newcommand{\Gal}{\text{Gal}}
\newcommand{\Br}{\textnormal{Br}}

\newcommand{\PP}{\mathbb{P}}

\newcommand{\mfI}{\mathfrak{I}}

\newcommand{\msS}{\mathscr{S}}
\newcommand{\mcZ}{\mathcal{Z}}
\newcommand{\mcP}{\mathcal{P}}

\numberwithin{equation}{section}

\title{Nontriviality of rings of integral-valued polynomials}

\author{Giulio Peruginelli\footnote{Department of Mathematics ``Tullio Levi-Civita'' University of Padova, Via Trieste, 63 35121 Padova, Italy. gperugin@math.unipd.it}
\and
Nicholas J. Werner\footnote{Department of Mathematics, Computer and Information Science, State University of New York at Old Westbury, Old Westbury, NY 11560, USA. wernern@oldwestbury.edu}}

\date{\today}

\begin{document}
\leftmark{\noindent  accepted for publication in Mathematische Nachrichten (2025).}
{\let\newpage\relax\maketitle}

\begin{abstract}

\noindent Let $S$ be a subset of $\oZ$, the ring of all algebraic integers. A polynomial $f \in \Q[X]$ is said to be integral-valued on $S$ if $f(s) \in \oZ$ for all $s \in S$. The set $\IntQ(S,\oZ)$ of all integral-valued polynomials on $S$ forms a subring of $\Q[X]$ containing $\Z[X]$. We say that $\IntQ(S,\oZ)$ is trivial if $\IntQ(S,\oZ) = \Z[X]$, and nontrivial otherwise. We give a collection of necessary and sufficient conditions on $S$ in order $\IntQ(S,\oZ)$ to be nontrivial. Our characterizations involve, variously, topological conditions on $S$ with respect to fixed extensions of the $p$-adic valuations to $\oQ$; pseudo-monotone sequences contained in $S$; ramification indices and residue field degrees; and the polynomial closure of $S$ in $\oZ$.\\

\noindent Keywords: integer-valued polynomial, algebraic integers, pseudo-monotone sequences, polynomial closure.

\noindent MSC Primary 12J20, 13F30, 13A18, 13B22, 13B25, 13F20.\\

\noindent\textbf{Data Availability Statement}\\
Data sharing is not applicable to this article. No new data was created or analyzed for this article.

\noindent\textbf{Conflict of Interest}\\
The authors declare that they have no conflicts of interest.
   \end{abstract}
\section{Introduction and Motivation}
\let\thefootnote\relax\footnotetext{The second author has been partially supported by the University of Padova.}
Let $\oQ$ be a fixed algebraic closure of the field of rational numbers, and let $\oZ$ be the integral closure of $\Z$ in $\oQ$, which is the ring of all algebraic integers. For a subset $S \subseteq \oZ$, a polynomial $f \in \Q[X]$ is said to be \textit{integral-valued} if $f(s) \in \oZ$ for all $s \in S$. When $S\subseteq\Z$, the definition reduces to the classical notion of integer-valued polynomials, i.e., $f(S)\subseteq\Z$. We adopt the general term `integral-valued' for polynomials with rational coefficients to underline the fact that the evaluation occurs at elements which are \textit{integral} over $\Z$. Sets of integral-valued polynomials often form rings with interesting properties, and these rings have been of interest in recent years \cite{ChabPer, HLP, LopWer, Mulay, Per1, PerWer}.

For each positive integer $n$, let $\An$ be the set of algebraic integers of degree at most $n$. In \cite{LopWer} (see also \cite{Per1} and \cite{PerWer}) the following ring of integral-valued polynomials was introduced in order to provide an example of a Pr\"ufer domain strictly contained between $\Z[X]$ and the classical ring of integer-valued polynomials $\Int(\Z)=\{f\in\Q[X] \mid f(\Z)\subseteq\Z\}$:
\begin{equation*}
\IntQ(\An):=\{f\in\Q[X] \mid f(\An)\subseteq\An\}.
\end{equation*}
Note that for $n=1$, we have $\Int(\Z)$. Clearly, $\An\subsetneq\mathcal{A}_{n+1}$ for each $n$. If $\alpha$ is an algebraic number of degree at most $n$ and $f\in\Q[X]$, then the degree of $f(\alpha)$ is also bounded by $n$. As we will later show in Lemma \ref{Psi and index}, there exist polynomials in $\Int_{\Q}(\mathcal{A}_{n+1})$ that are not in $\Int_{\Q}(\An)$. Thus, we have the following chain of strict inclusions (see Lemma \ref{Psi and index}):
\begin{equation}\label{An inclusion chain}
\ldots \subsetneq \Int_{\Q}(\mathcal{A}_{n+1}) \subsetneq \Int_{\Q}(\An) \subsetneq \Int_{\Q}(\mathcal{A}_{n-1}) \subsetneq \ldots \subsetneq \Int_{\Q}(\mathcal{A}_1)=\Int(\Z).
\end{equation}
By \cite[Theorem 3.9]{LopWer}, for each $n$ the ring $\IntQ(\An)$ is a Pr\"{u}fer domain. In particular, this means that the polynomial ring $\Z[X]$ is strictly contained in $\IntQ(\An)$ for each $n\geq 1$. Thus, for $n\geq2$, these rings provide examples of Pr\"ufer domains strictly contained between $\Z[X]$ and $\Int(\Z)$.

More generally, for $S \subseteq \oZ$ one may then consider the ring
\begin{equation*}
\IntQ(S, \oZ):=\{f\in\Q[X] \mid f(S)\subseteq \oZ \}.
\end{equation*}
Note that $\IntQ(\An) = \IntQ(\An,\oZ)$. It is stated in \cite[p.\ 2482]{LopWer} that $\IntQ(S,\oZ)$ lies properly between $\Z[X]$ and $\Int(\Z)$ for every subset $S$ of $\oZ$ properly containing $\Z$. The containment $\IntQ(S,\oZ) \subsetneq \Int(\Z)$ holds whenever $\Z \subsetneq S$ (see \cite{HLP}), but we demonstrate below in Example \ref{Motivating example} that $\Z[X]$ may equal $\IntQ(S,\oZ)$. Hence, the aforementioned statement from \cite[p.\ 2482]{LopWer} is false.\footnote{Thankfully, the fallacious claim from \cite{LopWer} was merely an expository remark, and does not affect the results of that paper.}

\begin{Def} For any $\alpha \in \oZ$, let $O_{\Q(\alpha)}$ be the ring of integers of $\Q(\alpha)$. A subset $S \subseteq \oZ$ is said to have \textit{unbounded degree} if the set $\{[\Q(\alpha):\Q] \mid \alpha \in S\}$ is unbounded. For $\alpha \in \oZ$, the \textit{index} of $\alpha$ is $\iota_\alpha := [O_{\Q(\alpha)} : \Z[\alpha]]$.
\end{Def}

\begin{Ex}\label{Motivating example} Let $S \subseteq \oZ$ be such that $S$ contains a sequence $\{\alpha_n\}_{n \in \N}$ of unbounded degree and $\iota_{\alpha_n}=1$ for all $n$. We will show that $\IntQ(S, \oZ) = \Z[X]$.

Certainly, $\Z[X] \subseteq \IntQ(S, \oZ)$. Let $f(X) \in \IntQ(S, \oZ)$ of degree $d$, and write $f(X) = \sum_{i=0}^d a_i X^i$, where each $a_i \in \Q$. Take an element $\alpha\in S$ of degree $n > d$, and consider $f(\alpha)$. By assumption $f(\alpha)$ is integral over $\Z$ and so it belongs to $O_{\Q(\alpha)}=\Z[\alpha]$. Thus,
\begin{equation*}
f(\alpha)=a_0+a_1\alpha+\ldots+a_d\alpha^d\in\Z[\alpha].
\end{equation*}
As a $\Z$-module, $\Z[\alpha]$ is a free with basis given by $1,\alpha,\ldots,\alpha^{n-1}$, so we also have
\begin{equation*}
f(\alpha)=b_0+b_1\alpha+\ldots+b_{n-1}\alpha^{n-1}
\end{equation*}
for some integers $b_0, \ldots, b_{n-1}$ that are uniquely determined by $f(\alpha)$. Since $n>d$ and every element of the $\Q$-vector space $\Q(\alpha)$ can be written uniquely as a $\Q$-linear combination of $1,\alpha,\ldots,\alpha^{n-1}$, each coefficients $a_i$ must be in $\Z$. Thus, $f(X) \in \Z[X]$ and consequently $\IntQ(S,\oZ) = \Z[X]$.
\end{Ex}

In particular, Example \ref{Motivating example} applies when $S = \oZ$ or when $S = \{\zeta_n\}_{n \in \N}$, where $\zeta_n$ is a primitive $n^\text{th}$ root of unity. Moreover, since $\IntQ(\oZ,\oZ) = \bigcap_{n \in \N} \IntQ(\An)$, Example \ref{Motivating example} demonstrates that the intersection of all the rings in \eqref{An inclusion chain} is equal to $\Z[X]$, which answers a question posed by David Dobbs to the first author in 2014. 

The goal of this paper is to characterize those subsets $S \subseteq \oZ$ such that $\IntQ(S,\oZ) \ne \Z[X]$. In pursuit of this problem, we will deal with more general rings of integer-valued polynomials, and describe when these rings are trivial, in the sense that they are equal to ordinary rings of polynomials. We refer to the papers \cite{PerWerNontrivial,Rush} for studies on related problems.

\begin{Def}\label{IVP notation}
Let $D$ be an integral domain, let $L$ be a field containing $D$, and let $F$ be a subfield of $L$. For each subset $S \subseteq D$, we define the following ring of integer-valued polynomials:
\begin{equation*}
\Int_F(S,D) := \{f \in F[X] \mid f(S) \subseteq D\}.
\end{equation*}
In the literature, when $F$ is the fraction field of $D$, the above ring is denoted by $\Int(S,D)$. Because our work will often involve changing the field of coefficients of these polynomials, we will always include the subscript $F$ in $\Int_F(S,D)$ for the sake of clarity, even in the case where $F$ is the fraction field of $D$.

It is clear that $(D \cap F)[X] \subseteq \Int_F(S,D)$. We say that $\Int_F(S,D)$ is \textit{trivial} if $(D \cap F)[X] = \Int_F(S,D)$, and \textit{nontrivial} otherwise.
\end{Def}

We seek to describe those subsets $S \subseteq \oZ$ for which $\IntQ(S,\oZ)$ is nontrivial. If $S$ is of bounded degree $n$, then $S \subseteq \An$, and we have $\Z[X] \subsetneq \IntQ(\An) \subseteq \IntQ(S,\Z)$. Thus, the sets of interest in this problem are all of unbounded degree. Example \ref{Motivating example} shows that if $S$ has unbounded degree and $\iota_s = 1$ for all $s \in S$, then $\IntQ(S,\oZ)$ is trivial. However, if we relax the condition on the indices $\iota_s$, then $\IntQ(S,\oZ)$ could be trivial or nontrivial.

\begin{Ex}\label{2Z example}
Let $S = 2\oZ = \{2\alpha \mid \alpha \in \oZ\}$. Then, $S$ has unbounded degree and not all of the indices $\iota_s$ are equal to 1. For this $S$, $X/2 \in \IntQ(S,\oZ)$, and hence $\IntQ(S,\oZ)$ is nontrivial. It is also possible to construct a set $S$ that has unbounded degree and with not all $\iota_s$ equal to 1, but for which $\IntQ(S,\oZ)$ is trivial; see Example \ref{unbounded trivial example 1}.
\end{Ex}

\begin{Ex}\label{nth roots of p example} Fix a prime $p$. For each $k \in \N$ let $a_k = 1 - \tfrac{1}{2^k}$, and take $S = \{p^{a_k}\}_{k \in \N} = \{p^{1/2}, p^{3/4}, p^{7/8}, \ldots\}$. Then, $S$ has unbounded degree, but $f(X) = X^2/p \in \IntQ(S,\oZ)$ because $f(p^{a_k}) = p^{a_{k-1}}$ for all $k \geq 2$.

This example can be generalized. Once again fix $p \in \PP$, but also fix a positive integer $n \geq 2$. For each $k \in \N$, let $b_k = \frac{1}{n-1}(1 - \frac{1}{n^k})$. Take $S = \{p^{b_k}\}_{k \in \N}$. Then, $f(X) = X^n/p \in \IntQ(S,\oZ)$, because $f(p^{b_1}) = 1$ and $f(p^{b_k}) = p^{b_{k-1}}$ for all $k \geq 2$.
\end{Ex}

In this paper, we provide both local and global characterizations of subsets $S \subseteq \oZ$ for which $\IntQ(S,\oZ)$ is nontrivial. We begin our study by considering integer-valued polynomial rings over subsets of valuation domains. Let $V$ be a valuation domain with fraction field $K$, and suppose $S \subseteq V$. In Section \ref{Val dom section}, we review the concept of pseudo-monotone sequences (as defined in \cite{ChabPolCloVal}), and we use these sequences to provide necessary and sufficient conditions on $S$ in order for $\Int_K(S,V)$ to be nontrivial (Theorem \ref{Int(S,V) nontrivial V1}). In Section \ref{Local section}, we return to the case where $S \subseteq \oZ$ and relate $\IntQ(S,\oZ)$ to rings of integer-valued polynomials over valuation domains in $p$-adic fields. In this way, Theorem \ref{Int(S,V) nontrivial V1} can be applied to $\IntQ(S,\oZ)$ (see Theorem \ref{Big thm local case} and Corollary \ref{Big cor local case}), and provides several characterizations of when $\IntQ(S,\oZ)$ is nontrivial.

In Section \ref{Global section}, we examine global conditions that imply that $\IntQ(S,\oZ)$ is trivial or nontrivial. 
We also consider how the degrees and indices of the algebraic integers in $S$ can influence $\IntQ(S,\oZ)$, present variations on Examples \ref{Motivating example} and \ref{nth roots of p example}. Section \ref{e and f section} examines similar questions with regard to ramification indices and residue field degrees in number fields generated by elements of $S$. We produce a new descending chain of Pr\"ufer domains between $\Z[X]$ and $\Int(\Z)$ by considering for each $n\in\N$ the compositum $\Q^{(n)}=\Q(\An)$ of all the number fields of degree bounded by $n$ and the ring $\IntQ(O_{\Q^{(n)}})$ of polynomials integral-valued over the ring of integers $O_{\Q^{(n)}}$ of this infinite algebraic extension (which is a non-Noetherian almost Dedekind domain with finite residue fields). Moreover, we review some classical examples due to Gilmer \cite[Example 14]{GilmEx}  and Chabert \cite[Example 6.2]{ChabEx} of an infinite algebraic extension $K$ of $\Q$ such that the integral closure $D$ in $K$ of $\Z_{(p)}$ for some prime $p\in\Z$ is an almost Dedekind domain with finite residue fields such that either the set of the residue field degrees (ramification indexes, respectively) of all the prime ideals of $D$ over $\Z_{(p)}$ are  unbounded and $\IntQ(D)$ is trivial. By means of the double boundedness condition of Loper \cite{LopIntD}, in Theorem \ref{nontrivial Prufer for rings} we prove that if $D$ is the integral closure of $\Z_{(p)}$ in some algebraic extension of $\Q$, then $\IntQ(D)$ is nontrivial if and only if $\IntQ(D)$ is Pr\"ufer.

The final section of the paper considers a broad generalization of Example \ref{2Z example}, describes the polynomial closure in $\oZ$ of a subset $S$ of the ring of all algebraic integers, and relates this concept to the nontriviality of $\IntQ(S,\oZ)$.

\section{Nontriviality over Valuation Domains}\label{Val dom section}
Throughout, $V$ is a valuation domain with valuation $v$, value group $\Gamma_v$, maximal ideal $M$, and fraction field $K$. We begin by recalling the  definitions of two kinds of pseudo-monotone sequences and related objects, which were introduced in \cite{ChabPolCloVal} in order to study the polynomial closure of subsets of a rank one valuation domain and studied further in \cite{PerPrufer, PS2}. These definitions originate from the notion of pseudo-convergent sequence given by Ostrowski in \cite{Ostrowski}, and we will not use this kind of sequence in this paper. After proving some basic properties of such sequences, we use them to characterize when $\Int_K(S,V)$ is nontrivial (see Theorem \ref{Int(S,V) nontrivial V1}).

\begin{Def}\label{pseudo-monotone def}
Let $\Lambda$ be an index set which we assume to be infinite and well-ordered, and let $E = \{s_i\}_{i \in \Lambda}$ be a sequence in $K$. Then, $E$ is said to be
\begin{enumerate}[(i)]
\item \textit{pseudo-divergent} if $v(s_{i}-s_j)>v(s_{j}-s_{k})$ for all $i<j<k\in\Lambda$;
\item \textit{pseudo-stationary} if $v(s_i-s_j)=v(s_{k}-s_{\ell})\in\Gamma_v$ for all $i\neq j \in \Lambda$, $k \neq \ell \in \Lambda$.
\end{enumerate}
A sequence that satisfies either of these two properties is said to be \textit{pseudo-monotone}.
\end{Def}

\begin{Def}\label{Gauge def} Let $E = \{s_i\}_{i \in \Lambda}$ be a pseudo-monotone sequence in $K$. We define the \emph{gauge} of $E$ as the following sequence $\{\delta_i\}_{i \in \Lambda}$ of $\Gamma_v$:
\begin{enumerate}[(i)]
\item if $E$ is pseudo-divergent, for each $i \in \Lambda$ we set $\delta_i=v(s_{i}-s_{j})$, with $j\in\Lambda$ and $j<i$;
\item if $E$ is pseudo-stationary, we let $\delta= v(s_{i}-s_{j})$ for any $i,j\in\Lambda$, $i \ne j$.
\end{enumerate}
\end{Def}

\begin{Def}\label{Breadth def} Each pseudo-monotone sequence $E = \{s_i\}_{i \in \Lambda}$ in $K$  has an associated \textit{breadth ideal}, which we denote by $\Br(E)$. The definition of $\Br(E)$ depends on which one of the two properties from Definition \ref{pseudo-monotone def} that $E$ satisfies.
\begin{enumerate}[(i)]
\item If $E$ is pseudo-divergent, then $\Br(E) = \{ x \in K \mid v(x) > \delta_{i} \text{ for some } i \in \Lambda\}$;
\item If $E$ is pseudo-stationary, then $\Br(E) = \{ x \in K \mid v(x) \geq \delta\}$.
\end{enumerate}
For a pseudo-divergent sequence $E$, it is possible that $\Br(E)=K$ (this occurs when the gauge of $E$ is coinitial in $\Gamma_v$). In all the other cases, $\Br(E)$ is a fractional $V$-ideal. When $V$ has rank one and $E=\{s_n\}_{n\in\N}$ is a pseudo-monotone sequence, the \textit{breadth} $\delta_E \in \R \cup \{\infty\}$ of $E$ is $\delta_E := \lim_{n \to \infty} v(s_{n+1}-s_n)$. In this case, $\delta_{i}$ can be replaced with $\delta_E$ in each of the above definitions for $\Br(E)$.
\end{Def}

Note that if $E$ is pseudo-stationary and $\delta = v(c)$ for some $c\in K$, then $\Br(E)=cV$ is a principal (fractional) ideal. Also, if $E \subset V$ is any pseudo-monotone sequence, then $\Br(E) \subseteq V$, and $\Br(E)\subseteq M$ if $E$ is pseudo-divergent.

Finally, we recall that, given $a\in K$ and $\delta\in\Gamma_v$, the ball of center $a$ and radius $\delta$ is $B(a,\delta)=\{x\in K\mid v(x-a)\geq\delta\}$. For each $b\in B(a,\delta)$ it is well known that $B(a,\delta)=B(b,\delta)$.

\begin{Lem}\label{delta balls lem}
Let $S \subseteq V$. Assume there exist a finite subset $T \subseteq S$ and $b \in M$ such that
for all $s \in S$, there exists $t \in T$ such that $v(s - t) \geq v(b)$.
\begin{enumerate}[(1)]
\item If $S$ contains a pseudo-divergent sequence $E$, then $\Br(E) \subseteq bM \subsetneq M$.
\item If $S$ contains a pseudo-stationary sequence $E$, then $\Br(E) \subseteq bV \subsetneq V$.
\end{enumerate}
\end{Lem}
Note that the assumption in this lemma is equivalent to saying that $S$ is covered by a finite union of balls $B(t,\delta)$, where $\delta > 0$ is fixed and each $t\in T$. 
\begin{proof}
Let $\delta = v(b) > 0$. Note that if either one of the two conditions is satisfied, then $S$ is infinite. In particular, $b\not=0$, because otherwise $S$ would be finite. Since $T$ is finite, there exists $t \in T$ such that $B(t,\delta)$ contains an infinite subsequence $E' = \{s_i\}_{i \in \Lambda'}$ of $E$. If $E$ is pseudo-divergent (respectively, pseudo-stationary), then $E'$ is also pseudo-divergent (resp., pseudo-stationary), and $\Br(E') = \Br(E)$. For any distinct $i, j \in \Lambda'$, we have $v(s_i - s_j) = v((s_i - t) + (t - s_j)) \geq \delta$.

When $E'$ is pseudo-divergent, the gauge of $E'$ is strictly decreasing, and the above inequality implies that $v(s_i - s_j) > \delta$. It follows that $\Br(E') \subseteq bM \subsetneq M$. If instead $E'$ is pseudo-stationary, then $v(s_i-s_j)$ is constant for all $i, j \in \Lambda'$. In this case, we have $v(s_i-s_j) \geq \delta > 0$, which means that $\Br(E') \subseteq bV \subsetneq V$.
\end{proof}

For a value group $\Gamma_v$ associated to the valuation $v$, the divisible hull of $\Gamma_v$ is $\Q\Gamma_v:=\Gamma_v \otimes_\Z \Q$. Note that if $L/K$ is an algebraic extension and $w$ is an extension of $v$ to $L$, then $w(\alpha) \in \Q\Gamma_v$ for all $\alpha \in L$.

\begin{Lem}\label{delta balls equivalence}
Let $S \subseteq V$. The following are equivalent.
\begin{enumerate}[(1)]
\item  There exist a finite subset $T \subseteq S$ and $\lambda \in \Q\Gamma_v\cup\{\infty\}$ such that $\lambda>0$ and for all $s \in S$, there exists $t \in T$ such that $v(s - t) \geq \lambda$.
\item There exist a finite subset $T \subseteq S$ and $\delta \in \Gamma_v\cup\{\infty\}$ such that $\delta>0$ and for all $s \in S$, there exists $t \in T$ such that $v(s - t) \geq \delta$.
\item There exists $b \in M$ such that $S/bV$ is finite.
\end{enumerate}
\end{Lem}
\begin{proof}
Note that $\lambda = \infty$ or $\delta=\infty$ in (1) or (2) corresponds to $b=0$ in (3). In this case, the set $S$ is finite. So, we assume throughout that $\lambda\ne \infty$ and $\delta \ne \infty$.


(1) $\Rightarrow$ (2) Assume that (1) holds. Let $\Gamma' = \{\gamma \in \Gamma_v \mid \gamma  \geq \lambda\}$. Suppose first that $\Gamma'$ contains a minimum element $\gamma'$. Then, for all $x \in V$,  $v(x) \geq \lambda$ if and only if $v(x) \geq \gamma'$, so we may take $\delta = \gamma'$ and (2) holds. So, assume that $\Gamma'$ has no minimum element. In particular, this means that $\lambda \notin \Gamma_v$.

Since $\lambda \in \Q\Gamma_v\setminus \Gamma_v$, there exists $\gamma \in \Gamma_v$ and $m \in \Z$, $m \geq 2$, such that $\lambda = \gamma/m$. For each $1 \leq k \leq m-1$, let 
\begin{equation*}
I_k = \{\varepsilon \in \Gamma' \mid k\lambda < \varepsilon \leq (k+1)\lambda\}.
\end{equation*}
We claim that there exists $k \in \{1, \ldots, m-1\}$ and $\varepsilon, \varepsilon' \in \Gamma'$ such that
\begin{equation}\label{lambda ineq}
k\lambda < \varepsilon < \varepsilon' \leq (k+1)\lambda.
\end{equation}
To see this, let $\delta_1 = \gamma = m\lambda \in \Gamma'$. Since $\Gamma'$ has no minimum value in $\Gamma_v$, there exist $\delta_2, \ldots, \delta_m \in \Gamma'$ such that $\delta_1 > \delta_2 > \cdots > \delta_m > \lambda$. Since there are $m-1$ sets $I_1, \ldots, I_{m-1}$, some $I_k$ must contain at least two distinct elements $\delta_i, \delta_j \in \{\delta_1, \ldots, \delta_m\}$ such that $\delta_i > \delta_j$. For this $k$, taking $\varepsilon=\delta_j$ and $\varepsilon'=\delta_i$ satisfies \eqref{lambda ineq}. Moreover, \eqref{lambda ineq} implies that $0 < \varepsilon'-\varepsilon < \lambda$. Thus, we may take $\delta = \varepsilon'-\varepsilon \in \Gamma_v$, and then for $x \in V$, $v(x) \geq \lambda$ implies that $v(x) \geq \delta > 0$. Hence, (2) holds.\\

(2) $\Rightarrow$ (3) Assume that (2) holds. Choose $b \in M$, such that $v(b) = \delta$. Let $T = \{t_1, \ldots, t_m\}$. We claim that $S/bV = \{t_i + bV \mid 1 \leq i \leq m\}$. Indeed, given $s \in S$, there exists $i$ such that $v(s-t_i) \geq \delta=v(b)$. So, $s-t_i \in bV$ and hence $s+bV = t_i + bV$.\\

(3) $\Rightarrow$ (1) Assume that (3) holds. Let $t_1, \ldots, t_m \in S$ be such that $S/bV = \{t_i + bV \mid 1 \leq i \leq m\}$. Take $T = \{t_1, \ldots, t_m\}$ and $\lambda = v(b) > 0$. Then, for any $s \in S$, there exists $i$ such that $s+bV = t_i +bV$, and for this $i$ we have $v(s-t_i) \geq v(b) = \lambda$.\end{proof}

\begin{Thm}\label{Int(S,V) nontrivial V1}
Let $S \subseteq V$. The following are equivalent.
\begin{enumerate}[(1)]
\item $\Int_K(S,V)$ is nontrivial.
\item There exist a finite subset $T \subseteq S$ and $\lambda \in \Q\Gamma_v\cup\{\infty\}$ such that $\lambda>0$ and for all $s \in S$, there exists $t \in T$ such that $v(s - t) \geq \lambda$.
\item  There exist a finite subset $T \subseteq S$ and $\delta\in\Gamma_v\cup\{\infty\}$ such that $\delta>0$ and for each $s \in S$, there exists $t \in T$ such that $v(s-t) \geq \delta$.
\item There exists $b \in M$ such that $S/bV$ is finite.
\item $S$ contains neither a pseudo-divergent sequence $E$ with $\Br(E) = M$, nor a pseudo-stationary sequence $E$ with $\Br(E) = V$.
\end{enumerate}
\end{Thm}
\begin{proof}
The theorem is trivial if $S$ is finite (which is equivalent to having $\lambda = \infty$ in (2), $\delta=\infty$ in (3), or $b=0$ in (4)), so we will assume throughout that $S$ is infinite.

(1) $\Rightarrow$ (2) Assume $V[X] \subsetneq \Int_K(S,V)$. Among all polynomials in $\Int_K(S,V) \setminus V[X]$, choose $f(X)$ of minimal degree $m$. Then, $m \geq 1$. Write $f(X) = g(X)/d$, where $g(X) \in V[X]$ and $d \in M$. Let $a$ be the leading coefficient of $g$.

We claim that $v(a) < v(d)$. If $v(a) \geq v(d)$, then $aX^m/d \in V[X]$, and $f(X) - aX^m/d$ would be either an element of $V[X]$, or an element of $\Int_K(S,V) \setminus V[X]$ of degree less than $m$. We reach a contradiction in either case. So, $v(a) < v(d)$.

Next, let $L/K$ be the splitting field of $g(X)$, and let $w$ be an extension of $v$ to $L$. In $L$, we have $g(X) = a(X-\alpha_1) \cdots (X-\alpha_m)$ for some $\alpha_1, \ldots, \alpha_m \in L$. Let $\lambda = (v(d)-v(a))/m \in \Q\Gamma_v$, which is greater than 0. Then, for all $s \in S$, there exists $1 \leq i \leq m$ such that $w(s-\alpha_i) \geq \lambda$. Indeed, if this is not the case, then for some $s \in S$,
\begin{equation*}
v(f(s)) = w(f(s)) < v(a) + m\lambda - v(d) = 0,
\end{equation*}
and $f \notin \Int_K(S,V)$.

For each $1 \leq i \leq m$, let $B_w(\alpha_i, \lambda) = \{x \in L \mid w(x-\alpha_i) \geq \lambda \}$ and let $B_i = B_w(\alpha_i, \lambda) \cap S$. Whenever $B_i$ has nonempty intersection with $S$, we may choose $t_i \in B_i \cap S$. 
Take $T = \{t_i \mid 1 \leq i \leq m \text{ and } B_i \cap S \ne \varnothing\}$. Note that $T\not=\varnothing$ because we showed above that $S\subseteq \bigcup_{i=1}^m B_w(\alpha_i,\lambda)$. Given $s \in S$, find $i$ such that $w(s-\alpha_i) \geq \lambda$. Then, since $w(t_i - \alpha_i) \geq \lambda$, we have $v(s-t_i) = w(s-t_i) \geq \lambda$. Thus, (2) holds. \\ 

(2) $\Leftrightarrow$ (3) $\Leftrightarrow$ (4) This is the content of Lemma \ref{delta balls equivalence}.\\

(4) $\Rightarrow$ (5) This follows from Lemma \ref{delta balls lem}.\\

(5) $\Rightarrow$ (1) We prove the contrapositive. Assume that $V[X] = \Int_K(S,V)$. Then, for all $m \geq 1$ and all $s_1, \ldots, s_m \in S$, and all $d \in M$, the polynomial $(X-s_1) \cdots (X-s_m)/d$ is not in $\Int_K(S,V)$. Consequently,
\begin{equation}\label{eq1}
\begin{aligned}
&\text{for all $s_1,\ldots,s_m \in S$, and all $\gamma \in \Gamma_v$ with $\gamma > 0$,}\\
&\text{there exists $t \in S$ such that $v(\textstyle\prod_{i=1}^m(t-s_i)) < \gamma$.}
\end{aligned}
\end{equation}

Suppose first that $v$ is discrete. By \eqref{eq1}, given any $s_1, \ldots, s_m \in S$, there exists $s_{m+1} \in S$ such that $v(\prod_{i=1}^m(s_{m+1}-s_i))=0$. So, we can construct a sequence $E = \{s_1, s_2, \ldots \}$ such that $v(s_i-s_j)=0$ for all $i \ne j$. This sequence is pseudo-stationary with $\Br(E)=V$.

For the remainder of the proof, we will assume that $v$ is not discrete. If $S$ contains a pseudo-stationary sequence $E$ with $\Br(E)=V$, then we are done. So, we will further assume that this is not the case. We will demonstrate how to construct a pseudo-divergent sequence $E \subseteq S$ with $\Br(E)=M$.

Choose $s_1 \in S$. For each $n \geq 1$, choose---if possible---$s_{n+1} \in S$ such that $v(\prod_{i=1}^n(s_{n+1}-s_i))=0$. We cannot do this indefinitely, since then $\{s_n\}_{n \in \N}$ would be pseudo-stationary with breadth ideal equal to $V$. So, there exists $N \geq 1$ such that 
\begin{equation}\label{eq2}
v(s_i-s_j) = 0 \text{ for all } 1 \leq i < j \leq N, \text{ and } v(\textstyle\prod_{i=1}^N (s - s_i)) > 0 \text{ for all } s \in S.
\end{equation}
In other words, for each $s\in S$, there exists a unique $i\in\{1,\ldots,N\}$ such that $v(s-s_i)>0$ and for every $j\in\{1,\ldots,N\}$, $j\not=i$, we have $v(s-s_j)=0$.

Let $\delta_1\in\Gamma_v,\delta_1>0$. By \eqref{eq1} and \eqref{eq2}, there exists $s_{N+1} \in S$ such that $0 < v(\prod_{i=1}^N(s_{N+1}-s_i)) < \delta_1$. By the previous paragraph, we have $v(s_{N+1}-s_{i_1}) > 0$ for a unique $i_1 \in \{1,\ldots,N\}$ and $v(s_{N+1}-s_j)=0$ for every $j \in \{1,\ldots,N\} \setminus \{i_1\}$. In particular, $0 < v(s_{N+1}-s_{i_1}) < \delta_1$. We choose now $\delta_2 \in \Gamma_v$ such that $0 < \delta_2 < v(s_{N+1}-s_{i_1})$ and correspondingly, by \eqref{eq1} and \eqref{eq2}, we can find $s_{N+2}\in S$ such that $0 < v(\prod_{i=1}^{N+1}(s_{N+2}-s_i)) < \delta_2$. As before, there exists a unique $i_2\in\{1,\ldots,N\}$ such that $v(s_{N+2}-s_{i_2})>0$ and $v(s_{N+2}-s_j)=0$ for every $j\in\{1,\ldots,N\}\setminus\{i_2\}$. 

At this point, we have 
\begin{equation*}
0 < v(s_{N+2}-s_{i_2}) < \delta_2 < v(s_{N+1}-s_{i_1}) < \delta_1.
\end{equation*}
If we continue in this way, we can produce a sequence $E=\{s_{N+k}\}_{k \in \N} \subseteq S$ and a sequence $\{\delta_k\}_{k \in \N} \subseteq \{\gamma \in \Gamma_v \mid \gamma > 0\}$ such that $\{\delta_k\}_{k \in \N}$ decreases to 0 and for each $k, \ell \in \N$, $k < \ell$ we have
\begin{equation}\label{eq3}
0 < v(s_{N+\ell} - s_{i_\ell}) < \delta_{\ell} < v(s_{N+k} - s_{i_k}) < \delta_{k}
\end{equation}
where $i_k, i_\ell \in \{1,\ldots,N\}$.  For each $1\leq i\leq N$, define 
\begin{equation*}
E_i:= \{s_{N+k} \mid k \in \N \text{ and } v(s_{N+k} - s_i) > 0\}.
\end{equation*}
Note that if $s_{N+k} \in E_i$, then $0 < v(s_{N+k}-s_i) < \delta_k$ and $v(s_{N+k}-s_j) = 0$ for every $j\in\{1,\ldots,N\}\setminus\{i\}$. 

Now, at least one of the sets $E_i$, $i=1, \ldots, N$, must be infinite. Without loss of generality, assume that $E_1$ is infinite. We will prove that $E_1$ is pseudo-divergent with $\Br(E_1) = M$. Write $E_1 = \{s_{k_0}, s_{k_1}, s_{k_2}, \ldots\}$, where $s_{k_0}=s_1$ and $k_1 < k_2 < \cdots$. By \eqref{eq3}, for $k_j > k_i$ we have
\begin{equation*}
0 < v(s_{k_j}-s_1) < \delta_{k_j} < v(s_{k_i}-s_1) < \delta_{k_i}
\end{equation*}
which shows that $E_1$ is pseudo-divergent with pseudo-limit $s_1$ (see for example \cite[Remark 4.7]{ChabPolCloVal}). Moreover, $\Br(E_1)=M$ because the sequence $\{\delta_k\}_{k \in \N}$ decreases to 0.
\end{proof}

\section{Local Case}\label{Local section}
In this section, we show how Theorem \ref{Int(S,V) nontrivial V1} can be used to decide whether or not $\IntQ(S,\oZ)$ is trivial. Let $\PP$ be the set of all prime numbers. For $p \in \PP$, $v_p$ is the standard $p$-adic valuation, $\Z_{(p)}$ is the localization of $\Z$ at $p\Z$, $\Z_p$ denotes the ring of $p$-adic integers, $\Q_p$ is the field of $p$-adic numbers, and $\mathbf{F}_p$ is the finite field with $p$ elements. We let $\oQp$ be an algebraic closure of $\Q_p$, and let $\oZp$ be the absolute integral closure of $\oZp$. Similarly, $\overline{\Z_{(p)}}$ denotes the integral closure of $\Z_{(p)}$ in $\oQ$. Note that $\oZp$ is a rank one non-discrete valuation domain with residue field equal to an algebraic closure of $\mathbb F_p$.

\begin{Def}\label{Sigma_p}
Let $p \in \PP$ and let $S \subseteq \overline{\Z_{(p)}}$. Let $\mathcal{P}(S)\subseteq\Z_{(p)}[X]$ be set of minimal polynomials over $\Q$ of all the elements of $S$. We define $\Sigma_p(S)$ to be the set of roots in $\oZp$ of the polynomials in $\mathcal{P}(S)$. 
\end{Def}

\begin{Thm}\label{Big thm local case}
Let $S \subseteq \oZ$. The following are equivalent.
\begin{enumerate}[(1)]
\item $\IntQ(S,\oZ)$ is nontrivial.
\item There exists $p \in \PP$ such that $\IntQ(S, \oQZp)$ is nontrivial.
\item There exists $p \in \PP$ such that $\Int_{\Q_p}(\Sigma_p(S), \oZp)$ is nontrivial.
\item There exists $p \in \PP$ such that $\Int_{\oQp}(\Sigma_p(S), \oZp)$ is nontrivial.
\end{enumerate}
\end{Thm}

The proof of Theorem \ref{Big thm local case} is quite long and requires numerous intermediate results. To keep things organized, we will prove the equivalences in the theorem one at a time. Before doing this, we mention the connection between Theorem \ref{Big thm local case} and our work from Section \ref{Val dom section}. Since $\oZp$ is a valuation domain with fraction field $\oQp$, combining Theorem \ref{Int(S,V) nontrivial V1} and Theorem \ref{Big thm local case} yields the following corollary.

\begin{Cor}\label{Big cor local case}
Let $S \subseteq \oZ$. Then, $\IntQ(S, \oZ)$ is nontrivial if and only if there exists $p \in \PP$ such that $\Sigma_p(S) \subseteq \oZp$ satisfies one of the equivalent conditions of Theorem \ref{Int(S,V) nontrivial V1}.
\end{Cor}

\subsection*{Equivalence of \ref{Big thm local case}(1) and \ref{Big thm local case}(2)}

First, we show that our rings of integral-valued polynomials are well-behaved with respect to localization at primes of $\Z$.

\begin{Lem}\label{AnpAn V1}
Let $p\in\PP$. Then, $\oQZp=\oZ_{(p)}$. Moreover, $(\mathcal{A}_{n})_{(p)}=\{\tfrac{\alpha}{t}\mid\alpha\in\mathcal{A}_n,t\in\Z\setminus p\Z\},$
and so $\bigcap_{p\in\PP}(\mathcal{A}_{n})_{(p)}=\mathcal{A}_{n}$.
\end{Lem}
\begin{proof}
This follows from the standard principle (see e.g.\ \cite[Proposition 5.12]{AtiyahMacDonald}) that integral closure and localization commute with one another.
\end{proof}

Since $\oZ=\bigcap_{p\in\PP}\oQZp$, we clearly have
\begin{equation}\label{globallocal}
\IntQ(S,\oZ)=\bigcap_{p\in\PP}\IntQ(S,\oQZp). 
\end{equation}
The next proposition shows that for each $p\in\PP$, $\Int(S,\oQZp)$ is the localization of $\Int(S,\oZ)$ at the multiplicative set $\Z\setminus p\Z$.

\begin{Prop}\label{localization IntQSZ V1}
Let $S\subseteq\oZ$ and let $p\in\PP$. Then, $\IntQ(S,\oZ)_{(p)}=\IntQ(S,\oQZp)$.
\end{Prop}
\begin{proof}
First, let $f\in \IntQ(S,\oZ)$ and $n\in\Z\setminus p\Z$. Then, for each $s\in S$, we have that $\frac{f(s)}{n}$ is integral over $\Z_{(p)}$ by Lemma \ref{AnpAn V1}. Thus, $\IntQ(S,\oZ)_{(p)} \subseteq \IntQ(S,\oQZp)$.

The proof of the other containment is similar to the proof of \cite[(7)]{ChabPerAdelic}. Let $f\in \IntQ(S,\oQZp)$. There exists $d\in\Z,d\not=0$ such that  $df(X)\in\Z[X]$. Suppose that $d=p^a t$, for some $a\geq0$ and $t\in\Z$ not divisible by $p$. Let $g(X)=t f(X)$. Then $g\in \IntQ(S,\oQZp)$. Moreover, for each prime $q\not=p$, $g\in \Z_{(q)}[X]\subseteq\IntQ(S,\overline{\Z_{(q)}})$. Now, we have
\begin{equation*}
g\in\bigcap_{p\in\PP}\IntQ(S,\oQZp)=\IntQ(S,\oZ)
\end{equation*}
so that $f=\frac{g}{t}$ is in $\IntQ(S,\oZ)_{(p)}$ because $t$ is invertible in $\Z_{(p)}$.
\end{proof}

\begin{Prop}\label{1 iff 2}
Let $S\subseteq\oZ$. Then $\IntQ(S,\oZ)$ is trivial if and only if, for each $p\in\PP$, $\IntQ(S,\oQZp)$ is trivial.
\end{Prop}
\begin{proof}
The `if' direction follows by \eqref{globallocal} and the fact that $\Z[X]=\bigcap_{p\in\PP}\Z_{(p)}[X]$. Conversely, suppose that $\IntQ(S,\oZ)=\Z[X]$. Localizing at any $p\in\PP$, by Proposition \ref{localization IntQSZ V1} we get $\IntQ(S,\oZ)_{(p)}=\Z_{(p)}[X]=\IntQ(S,\oQZp)$.
\end{proof}

\subsection*{Equivalence of \ref{Big thm local case}(2) and \ref{Big thm local case}(3)}
\begin{Prop}\label{IntAnpcompletion}
Let $p \in \PP$ and let $S \subseteq \oQZp$. Then, $\IntQ(S,\oQZp) = \IntQ(\Sigma_p(S),\oZp)$. 
\end{Prop}
\begin{proof}
First, assume $f \in \IntQ(S,\oQZp)$ and let $\alpha \in \Sigma_p(S)$. Then, by definition of the set $\Sigma_p(S)$, there exists $\beta \in S$ such that $\alpha$ and $\beta$ have the same minimal polynomial over $\Q$. In particular, $\alpha = \sigma(\beta')$, where $\beta'\in\oQ$ is a conjugate of $\beta$ over $\Q$ and $\sigma:\Q(\beta') \to \Q_p(\alpha)$ is a $\Q$-embedding (which corresponds to some prime ideal $P$ of $O_{\Q(\beta')}$ above $p$). By assumption, $f(\beta)$ is integral over $\Z_{(p)}$; but, $f$ is also integral-valued on $\beta'$, even if $\beta'$ might not lie in $S$. Hence, $f(\alpha) = \sigma(f(\beta')) \in O_{\Q_p(\alpha)} \subseteq \oZp$.

Conversely, suppose $f \in \IntQ(\Sigma_p(S), \oZp)$ and let $\beta \in S$. For each $\Q$-embedding $\sigma: \Q(\beta) \to \oQp$, we have $\sigma(\beta) \in \Sigma_p(S)$. Thus, $f(\sigma(\beta)) \in O_{\Q_p(\sigma(\beta))}$. Denote the localization of $O_{\Q(\beta)}$ at a prime $P$ by $O_{\Q(\beta), P}$. Then,
\begin{equation*}
f(\beta) \in \bigcap_{P|p} O_{\Q(\beta), P} = O_{\Q(\beta),(p)}.
\end{equation*}
Since $O_{\Q(\beta),(p)} \subseteq \overline{\Z_{(p)}}$, we obtain the desired result.
\end{proof}



\begin{Prop}\label{2 iff 3}
Let $p \in \PP$ and $S \subseteq \oQZp$. Then $\IntQ(S,\oQZp)$ is nontrivial if and only if $\Int_{\Q_p}(\Sigma_p(S),\oZp)$ is nontrivial. 
\end{Prop}
\begin{proof}
Note that by Proposition \ref{IntAnpcompletion},  we have $\IntQ(S,\oQZp)=\IntQ(\Sigma_p(S),\oZp)$ so that
\begin{equation*}
\IntQ(S,\oQZp)=\Int_{\Q_p}(\Sigma_p(S),\oZp)\cap\Q[X].
\end{equation*}
Now, $\oZp \cap \Q_p = \Z_p$ and $\Z_p \cap \Q = \Z_{(p)}$. So, if $\IntQ(S,\oQZp)$ is nontrivial, then $\Int_{\Q_p}(\Sigma_p(S),\oZp)$ is nontrivial.

Conversely, suppose that there exists $f\in\Int_{\Q_p}(\Sigma_p(S),\oZp)\setminus\Z_p[X]$. Since $\Q_p=\Z_p[\frac{1}{p}]$, we may write
\begin{equation*}
f(X)=\frac{F(X)}{p^n}=\frac{\sum_{i=0}^d \alpha_i X^i}{p^n}
\end{equation*}
where $F(X)=\sum_i\alpha_iX^i\in\Z_p[X]$ and $n\geq1$. For each $i=0, \ldots, d$, choose $a_i \in \Z_{(p)}$ such that $v_p(a_i - \alpha_i) \geq n$. We consider the polynomial
\begin{equation*}
g(X)=\frac{G(X)}{p^n}=\frac{\sum_{i=0}^d a_i X^i}{p^n}.
\end{equation*}
By construction, $g-f \in \Z_p[X]$, so $g \in \Q[X]\setminus \Z_{(p)}[X]$. We will show that $g\in \IntQ(\Sigma_p(S),\oZp)$. Let $s\in \Sigma_p(S)$. Then,
\begin{equation*}
g(s)=\frac{G(s)-F(s)+F(s)}{p^n}=\frac{G(s)-F(s)}{p^n}+f(s)
\end{equation*}
For each $0\leq i \leq d$, we have $v_p(a_i-\alpha_i)+iv(s)\geq n$, so  $\frac{G(s)-F(s)}{p^n}$ has non-negative valuation. Thus, $g(s) \in \oZ_p$.
\end{proof}

\subsection*{Equivalence of \ref{Big thm local case}(3) and \ref{Big thm local case}(4)}
For a prime $p$, we denote by $G_p$ the absolute Galois group of $\Q_p$, that is, $G_p=\Gal(\oQp/\Q_p)$. For a subset $S$ of $\oZp$ we set
\begin{equation*}
G_p(S)=\{\sigma(s)\mid s\in S,\sigma\in G_p\},
\end{equation*}
which is a subset of $\oZp$ consisting of the union of all the conjugates of the elements in $S$. Note that we have the equality $\Int_{\Q_p}(S,\oZp)=\Int_{\Q_p}(G_p(S),\oZp)$.

\begin{Prop}\label{nontriviality Galois conjugate}
Let $p \in \PP$ and let $S\subseteq\oZp$. The following are equivalent.
\begin{enumerate}[(1)]
\item $\Int_{\Q_p}(S,\oZp)$ is nontrivial.
\item $\Int_{\oQp}(G_p(S),\oZp)$ is nontrivial.
\item $\Int_{\oQp}(S,\oZp)$ is nontrivial.
\end{enumerate}
\end{Prop}
\begin{proof}
(1) $\Leftrightarrow$ (2) By \cite[Lemma 2.20]{PerDedekind}, the integral closure of $\Int_{\Q_p}(S,\oZp)=\Int_{\Q_p}(G_p(S),\oZp)$ in $\oQp(X)$ is equal to $\Int_{\oQp}(G_p(S),\oZp)$. In particular, $\Int_{\Q_p}(S,\oZp)$ strictly contains $\Z_p[X]$ if and only if $\Int_{\oQp}(G_p(S),\oZp)$  strictly contains $\oZp[X]$. 

(2) $\Rightarrow$ (3) Since $S\subseteq G_p(S)$, we have $\Int_{\oQp}(G_p(S),\oZp) \subseteq \Int_{\oQp}(S,\oZp)$. So, (2) implies (3).

(3) $\Rightarrow$ (2) Assume $\Int_{\oQp}(S,\oZp)$ is nontrivial. By Theorem \ref{Int(S,V) nontrivial V1}, there exists a finite set $T \subseteq S$ and $\delta > 0$ such that for each $s \in S$, there exists $t \in T$ with $v(s-t) \geq \delta$. Note that since $T$ is finite, so is the set $G_p(T)$ of all images of $T$ under $G_p$. Let $\sigma(s) \in G_p(S)$, where $\sigma \in G_p$ and $s \in S$. Then, there exists $t \in T$ with $v(s-t) \geq \delta$. Since $\oQ_p$ is Henselian, we have $v(\sigma(\alpha)) = v(\alpha)$ for all $\alpha \in \oQ_p$. So,
\begin{equation*}
v(\sigma(s) - \sigma(t)) = v(\sigma(s-t)) = v(s-t) \geq \delta.
\end{equation*}
Thus, condition (2) of Theorem \ref{Int(S,V) nontrivial V1} holds for $G_p(S)$, and we conclude that $\Int_{\oQ_p}(G_p(S),\oZp)$ is nontrivial.
\end{proof}

\section{Global Conditions and Examples}\label{Global section}

Here, we examine global conditions that can be used to decide when $\IntQ(S,\oZ)$ is nontrivial. Our first theorem of this type relates $\IntQ(S,\oZ)$ to some distinguished polynomials that lie in the rings $\IntQ(\An)$.

\begin{Def}\label{BCL poly} For each $p \in \PP$ and each positive integer $n \geq 1$, let \begin{equation*}
\Psi_{p,n}(X) = (X^{p^n}-X)(X^{p^{n-1}}-X) \cdots (X^p-X).
\end{equation*}
\end{Def}

Recall that when $\alpha \in \oZ$, the index of $\alpha$ is $\iota_\alpha = [O_{\Q(\alpha)} : \Z[\alpha]]$. The lemma below summarizes the basic relationships among $\Psi_{p,n}$, $\iota_\alpha$, and integral-valued polynomials.

\begin{Lem}\label{Psi and index}
Let  $n\in\N$ and $p\in\PP$.
\begin{enumerate}[(1)]
\item (\cite[Theorem 3]{BCL}) Modulo $p$, $\Psi_{p,n}$ is the monic least common multiple of all polynomials in $\F_p[X]$ of degree at most $n$.
\item For each prime $p$, $\Psi_{p,n}/p \in \IntQ(\An) \setminus \IntQ(\mathcal{A}_{n+1})$.
\item Let $\alpha \in \oZ$ and let $f \in \IntQ(\{\alpha\}, \oZ)$. If $\deg f < [\Q(\alpha): \Q]$, then $\iota_\alpha f \in \Z[X]$. In particular, $f\in\Z_{(q)}[X]$ for each prime $q$ not dividing $\iota_{\alpha}$.
\end{enumerate}
\end{Lem}
\begin{proof}
For (2), if $\alpha \in \An$ with minimal polynomial $m_\alpha(X)$, then $m_\alpha \mid \Psi_{p,n}$ modulo $p$. Hence, $\Psi_{p,n}(\alpha) \in pO_{\Q(\alpha)}$. To show that $\Psi_{p,n}/p \notin \IntQ(\mathcal{A}_{n+1})$, let $f \in \F_p[X]$ be monic and irreducible of degree $n+1$. Then, $f \nmid \Psi_{p,n}$ mod $p$. Next, let $F \in \Z[X]$ be monic, irreducible, and such that $F \equiv f$ mod $p$ (the existence of $F$ follows from Perron's Criterion for irreducibility \cite[Theorem 2.2.5]{Pra}). Let $\beta \in \oZ$ be a root of $F$, and let $O = O_{\Q(\beta)}$; note that $\beta\in\mathcal A_{n+1}$. Suppose that $\Psi_{p,n}(\beta)/p \in \oZ$. Then, $\Psi_{p,n}(\beta) \in pO$, and hence $\Psi_{p,n}(\beta) \equiv 0$ mod $p$. Since the image of $\beta$ in $O/pO$ is a root of $f$, this means that $f \mid \Psi_{p,n}$ mod $p$, which is a contradiction.

For (3), let $N = [\Q(\alpha): \Q]$ and assume that $f(X) = \sum_{i=0}^d a_i X^i$, where each $a_i \in \Q$ and $d < N$. Then, $f(\alpha) \in \oZ \cap \Q(\alpha) =  O_{\Q(\alpha)}$. Now, $\iota_\alpha O_{\Q(\alpha)} \subseteq \Z[\alpha]$, so $\sum_{i=0}^d \iota_\alpha a_i \alpha^i \in \Z[\alpha]$. Since $\Z[\alpha]$ is a free $\Z$-module with basis $1, \alpha, \ldots, \alpha^{N-1}$ and $d < N$, each $\iota_\alpha a_i \in \Z$, as required.
\end{proof}

In the examples later in this section, we will sometimes use extensions of $p$-adic valuations to $\oQ$ in order to prove that $\IntQ(S,\oZ)$ is trivial. The next two lemmas demonstrate how to do this.

\begin{Lem}\label{Extensions to Qbar lemma 1}
Let $S \subseteq \oZ$.
\begin{enumerate}[(1)]
\item Let $\sigma\in G={\rm Gal}(\oQ/\Q)$ and $U$ be a valuation domain of $\oQ$ extending $\Z_{(p)}$ for some $p\in\PP$. Then $\IntQ(S,U)=\IntQ(\sigma(S),\sigma(U))$. In particular, if $S$ is $G$-invariant (i.e., $\sigma(S)=S,\forall\sigma\in G$) and $U,U'$ are two valuation domains extending $\Z_{(p)}$ then $\IntQ(S,U)=\IntQ(S,U')$ and so $\IntQ(S,\overline{\Z_{(p)}})=\IntQ(S,U)$ for any valuation domain $U$ extending $\Z_{(p)}$.

\item Let $p \in \PP$. If there exists an extension $u$ of $v_p$ to $\oQ$ with associated valuation domain $U$ such that $\Int_{\oQ}(S,U)$ is trivial, then $\IntQ(S,\overline{\Z_{(p)}}) = \Z_{(p)}[X]$.
\item If, for every $p \in \PP$, there exists an extension $u_p$ of $v_p$ to $\oQ$ with associated valuation domain $U_p$ such that $\Int_{\oQ}(S,U_p)$ is trivial, then $\IntQ(S,\oZ)$ is trivial.
\end{enumerate}
\end{Lem}
\begin{proof}
(1)  Let $f\in\IntQ(S,U)$. Since $f \in \Q[X]$, $\sigma(f(s))=f(\sigma(s))\in \sigma(U)$ for each $s \in S$, so $f\in \IntQ(\sigma(S),\sigma(U))$. The other containment is proved in the same way by applying $\sigma^{-1}$ to the previous relation. Now, suppose that $S$ is $G$-invariant. The valuation domains $U,U'$ are conjugate by some $\sigma\in G$, so $U'=\sigma(U)$ and $\IntQ(S,U)=\IntQ(S,U')$. For the final claim, note that $\overline{\Z_{(p)}}=\bigcap_{u|v_p}U$, where the the intersection is taken over all extensions $u$ of $v_p$ to $\oQ$, and $U$ is the valuation domain corresponding to $u$. We have,
\begin{equation}\label{IntQSZp}
\IntQ(S,\overline{\Z_{(p)}})=\bigcap_{u|v_p}\IntQ(S,U)    
\end{equation}
and so if $S$ is $G$-invariant, we have $\IntQ(S,\overline{\Z_{(p)}})=\IntQ(S,U)$ for each $u$ extending $v_p$ to $\oQ$.

(2) Assume the desired $u$ and $U$ exist. Then, $\Int_{\oQ}(S,U) = U[X]$, so
\begin{equation*}
\IntQ(S,U) = \Int_{\oQ}(S,U) \cap \Q[X] = \Z_{(p)}[X].
\end{equation*}
By \eqref{IntQSZp}, in order for $\IntQ(S,\overline{\Z_{(p)}})$ to be trivial it is sufficient that $\IntQ(S,U)=\Z_{(p)}[X]$ for only one extension $u$ of $v_p$. As shown above, this holds for $u$, so we conclude that $\IntQ(S,\overline{\Z_{(p)}}) = \Z_{(p)}[X]$.

(3) By (2), we have $\IntQ(S, \overline{\Z_{(p)}}) = \Z_{(p)}[X]$ for each $p$, and then $\IntQ(S,\oZ) = \Z[X]$ by \eqref{globallocal}.
\end{proof}

\begin{Rem}
We stress that the condition appearing in item (2) of Lemma \ref{Extensions to Qbar lemma 1} is sufficient but not necessary, if $S$ is not $G$-invariant ((1) of the same Lemma). We will see in Example \ref{Gilmer example} (due to Gilmer, \cite[Example 14]{GilmEx}) and in Example \ref{Chabert example}  (due to Chabert, \cite[Example 6.2]{ChabEx}) that there are subsets $S\subset \overline{\Z_{(p)}}$ such that $\IntQ(S,U)$ is non trivial for each extension $U$ of $\Z_{(p)}$ to $\oQ$ but their intersection $\bigcap_{U|\Z_{(p)}}\IntQ(S,U)=\IntQ(S,\overline{\Z_{(p)}})$ \eqref{IntQSZp} is trivial (see Propositions \ref{Unbounded res field deg prop} and  \ref{triviality unbounded ramification}).
\end{Rem}

\begin{Lem}\label{Extensions to Qbar lemma 2}
Let $S =\{s_i\}_{i \in \N} \subseteq \oZ$. Assume that $p \in \PP$ and $u$ is an extension of $v_p$ to $\oQ$ such that one of the following two conditions holds:
\begin{enumerate}[(i)]
\item $\{u(s_i)\}_{i \in \N}$ is a strictly decreasing sequence with limit 0.
\item $u(s_i-s_j)=0$ for all distinct $i, j \in \N$.
\end{enumerate}
Then, $\IntQ(S,\overline{\Z_{(p)}}) = \Z_{(p)}[X]$. Moreover, if for every $p \in \PP$ there exists an extension $u$ of $v_p$ satisfying either (i) or (ii), then $\IntQ(S,\oZ)$ is trivial.
\end{Lem}
\begin{proof}
Let $U \subseteq \oQ$ be the valuation domain associated to $u$, and let $M$ be the maximal ideal of $U$. If (i) holds, then with respect to $u$, $S$ is a pseudo-divergent sequence with breadth ideal $M$; and if (ii) holds, then $S$ is pseudo-stationary with respect to $u$, and the breadth ideal is $U$. In either case, $\Int_{\oQ}(S,U)$ is trivial by Theorem \ref{Int(S,V) nontrivial V1}. By Lemma \ref{Extensions to Qbar lemma 1}, $\IntQ(S,\overline{\Z_{(p)}}) = \Z_{(p)}[X]$, and if this is true for every prime $p$, then $\IntQ(S,\oZ)$ is trivial by \eqref{globallocal}.
\end{proof}

Next, we will explore examples of unbounded sets $S \subseteq \oZ$ such that $\IntQ(S,\oZ)$ is trivial. Later in Sections \ref{e and f section} and \ref{Poly closure section}, we construct unbounded sets for which $\IntQ(S,\oZ)$ is nontrivial. As demonstrated in Example \ref{Motivating example}, the indices of algebraic integers can be used to give a sufficient condition for $\IntQ(S,\oZ)$ to be trivial. We now restate that example, and give a slight variation that does not require each index $\iota_s$ to be 1.

\begin{Prop}\label{index trivial conditions}
Let $S \subseteq \oZ$.
\begin{enumerate}[(1)]
\item Assume that $S$ contains a sequence $\{s_i\}_{i \in \N} \subseteq S$ of unbounded degree and such that $\iota_{s_i}=1$ for all $i \in \N$. Then, $\IntQ(S,\oZ)$ is trivial.

\item Assume that for each integer $n \geq 1$, there exists a finite subset $\{s_1, \ldots, s_m\} \subseteq S$ such that $[\Q(s_i):\Q] > n$ for each $1 \leq i \leq m$ and $\gcd(\iota_{s_1}, \ldots, \iota_{s_m})=1$. Then, $\IntQ(S,\oZ)$ is trivial.
\end{enumerate}
\end{Prop}
\begin{proof}
Part (1) is Example \ref{Motivating example}. For (2), let $f \in \IntQ(S, \oZ)$ of degree $d$ and let $s_1, \ldots, s_m$ be the elements of $S$ with the stated properties. For each $i$, we have $[\Q(s_i):\Q] > d$, so $\iota_{s_i}f \in \Z[X]$ by Lemma \ref{Psi and index}(3). The condition $\gcd(\iota_{s_1}, \ldots, \iota_{s_m}) = 1$ then implies that $f \in \Z[X]$.
\end{proof}

 In Example \ref{unbounded trivial example 1}, we will show that the conditions on indices in Proposition \ref{index trivial conditions} are not necessary for $\IntQ(S,\oZ)$ to be trivial. This example requires Dedekind's Index Theorem, which describes the primes $p$ such that $p \mid \iota_\alpha$. For this theorem, we will follow the treatment given in \cite{Conrad}. An equivalent statement is available in \cite[Theorem 6.1.4]{Cohen}.

\begin{Thm}\label{Dedekind index theorem}(Dedekind Index Theorem)
Let $\alpha$ be an algebraic integer with minimal polynomial $f \in \Z[X]$. For a prime $p$, use a bar to denote reduction mod $p$. Factor $f$ mod $p$ as
\begin{equation*}
\overline{f} = \overline{\pi_1}^{e_1} \cdots \overline{\pi_k}^{e_k},
\end{equation*}
where each $\overline{\pi_i} \in \F_p[X]$ is monic and irreducible, and each $e_i \geq 1$. For each $i$, let $\pi_i$ be a monic lift of $\overline{\pi_i}$ to $\Z[X]$ and let $F \in \Z[X]$ be such that
\begin{equation*}
f = \pi_1^{e_1} \cdots \pi_k^{e_k} + pF.
\end{equation*}
Then, $p \mid \iota_\alpha$ if and only if $\overline{\pi_j} \mid \overline{F}$ in $\F_p[X]$ for some $1 \leq j \leq k$ such that $e_j \geq 2$.
\end{Thm}

\begin{Def}\label{c,n polys} Let $c, n \in \Z$ be such that $c \geq 2$ and $n \geq 2$. We define $$f_{c,n}(X):= X^n + c^3 X^{n-1} + c^2$$ and $\iota_{c,n} := \iota_\alpha$, where $\alpha$ is a root of $f_{c,n}$.
\end{Def}

\begin{Lem}\label{c^3 polys}
Let $c, n \in \Z$ be such that $c \geq 2$ and $n \geq 2$.
\begin{enumerate}[(1)]
\item The polynomial $f_{c,n}$ is irreducible over $\Q$.

\item Let $\alpha\in\oQ$ be a root of $f_{c,n}$ and $u$ an extension of the $p$-adic valuation $v_p$ to $\oQ$. If $p$ is a prime and $p \mid c$, then $u(\alpha) = 2v_p(c)/n$.

\item If $p$ is a prime and $p \mid c$, then $p \mid \iota_{c,n}$.

\item If $q$ is a prime and $q \nmid c$ but $q \mid n$, then $q \nmid \iota_{c,n}$.

\item If $m\in\N$ and $q$ is a prime such that $q \nmid c$, then there exists a prime $\ell>m$ such that $q \nmid \iota_{c,\ell}$.
\end{enumerate}
\end{Lem}
\begin{proof}
(1) This follows from Perron's Criterion for irreducibility \cite[Theorem 2.2.5]{Pra}.

(2) We have $\alpha^n = -c^2(c\alpha^{n-1}+1)$. If $p$ divides $c$, then $u(\alpha^n) = v_p(c^2)$, and so $u(\alpha) = 2v_p(c)/n$.

(3) Assume $p \mid c$, and use a bar to denote reduction mod $p$. Then, $\overline{f_{c,n}}(X) = X^n$, and $f_{c,n}(X) = X^n + pF_{c,n}(X)$, where $F_{c,n}(X) = \tfrac{c^3}{p}X^{n-1} + \tfrac{c^2}{p}$. Note that $F_{c,n} \in \Z[x]$ and $\overline{F_{c,n}}= 0$. By Theorem \ref{Dedekind index theorem}, $p \mid \iota_{p,n}$.

(4) Assume $q \nmid c$ and $q \mid n$. When reduced mod $q$, the derivative of $f_{c,n}(X)$ equals $(n-1)c^3X^{n-2}$, which is either constant or has 0 as its only root. Since $q \nmid f_{c,n}(0)$, we see that $f_{c,n}$ mod $q$ is separable. Hence, when applying the Dedekind Index Theorem, each exponent $e_j$ will equal 1. We conclude that $q \nmid \iota_{c,n}$.

(5) We have $f_{c,n}'(X)=X^{n-2}((n-1)c^3+nX)$. By the Dirichlet's theorem on arithmetic progressions, there exists a prime $\ell$ of the form $qk+1$, $\ell>m$; in particular,  $v_q(\ell-1)>0$ and so, modulo $q$, the only root of $f_{c,\ell}'(X)$ is $0$ (note that $v_q(\ell)=0$) which is not a root of $f_{c,\ell}$ modulo $q$ because we are assuming $v_q(c)=0$. Hence, $f_{c,\ell}$ is separable modulo $q$ so that $q\nmid\iota_{c,\ell}$ by Theorem \ref{Dedekind index theorem}.
\end{proof}

\begin{Ex}\label{unbounded trivial example 1}
It is possible that the indices of the elements of an unbounded sequence $S=\{s_i\}_{i\in\N}\subseteq\oZ$ can all share a common prime factor, yet the corresponding ring of integral-valued polynomials $\IntQ(S,\oZ)$ is trivial.

Fix a prime $p$, and let $q_1 < q_2 < \ldots$ be an ordering of all the primes not equal to $p$. For each $i \in \N$, let $f_{p,q_i}$ be as in Definition \ref{c,n polys}, and let $s_i\in\oQ$ be a root of $f_{p,q_i}$. Take $S = \{s_i\}_{i \in \N}$. Then, $\{[\Q(s_i): \Q]\}_{i \in \N}$ is unbounded, and by Lemma \ref{c^3 polys} each $\iota_{s_i}$ is divisible by $p$ and not divisible by $q_i$. We will show that $\IntQ(S,\oZ)$ is trivial.

By Lemma \ref{c^3 polys}(2), $u_p(s_i)=2/q_i$ for each $i$, so $\{u_p(s_i)\}_{i \in \N}$ is a strictly decreasing sequence with limit $0$. By Lemma \ref{Extensions to Qbar lemma 2}, $\IntQ(S, \overline{\Z_{(p)}}) = \Z_{(p)}[X]$.

Let now $q$ be a prime different from $p$. Given  $f\in\IntQ(S,\oZ)$, by Lemma \ref{c^3 polys}(5), there exists a prime $q_i$ such that $[\Q(s_i):\Q]=q_i>\deg(f)$ and $q\nmid\iota_{s_i}$. By Lemma \ref{Psi and index}(3), $\iota_{s_i}\cdot f\in\Z[X]$ so that $f$ belongs to $\Z_{(q)}[X]$. This argument shows that $\IntQ(S,\overline{\Z_{(q)}})=\Z_{(q)}[X]$ for every prime $q\not=p$.

By \eqref{globallocal}, we conclude that $\IntQ(S,\oZ)=\Z[X]$.
\end{Ex}

In the following examples, for each $p\in\PP$, we fix an extension $u_p$ of $v_p$ to $\oQ$ with corresponding valuation domain $U_p$ having maximal ideal $\oMp$. For each positive integer $n$, let $\zeta_n$ be a primitive $n^\text{th}$ root of unity. Sequences of these algebraic integers can lead to interesting examples of sets $S$ for which $\IntQ(S,\oZ)$ is trivial. In the next three examples, we will construct a sequence $S\subseteq\oZ$ that is pseudo-divergent with $\Br(S)=\oMp$ with respect to $u_p$ for a single prime $p$, but pseudo-stationary with $\Br(S)=U_q$ at all other primes $q$; a sequence $S$ that is pseudo-stationary with $\Br(S)=U_p$ with respect to all primes; and a sequence $S$ which is eventually pseudo-divergent with $\Br(S)=\oMp$ with respect to each prime.

\begin{Ex}\label{Root of unity example}
Let $p \in \PP$ and take $S=\{\zeta_{p^k}\}_{k\in\N}$. Then, $\IntQ(S,\oZ)$ is trivial by Proposition \ref{index trivial conditions}(1). We claim that $S$ is a pseudo-divergent sequence with $\Br(S)=M_p$ with respect to $u_p$ and it is pseudo-stationary with $\Br(S)=U_q$ with respect to $u_q$ for every prime $q\not=p$.

Whenever $j < k$, we have $\zeta_{p^k}-\zeta_{p^j}=\zeta_{p^k}(1-\zeta_{p^k}^{m})$, where $m=p^{k-j}-1$. Since $m$ and $p^k$ are coprime, $\zeta_{p^k}^m$ is also a primitive $p^k$-th root of unity. So,
\begin{equation*}
u_q(\zeta_{p^k}-\zeta_{p^j})=u_q(1-\zeta_{p^k}^m) = u_q(1-\zeta_{p^k}).
\end{equation*}
We recall that the prime $p$ is totally ramified in $\Q(\zeta_{p^k})$ for every $k\in\N$ (see for example \cite[Theorem 26]{Marcus}), and  $u_p(1-\zeta_{p^{k}})=(p^{k-1}(p-1))^{-1}$ by \cite[Chapter 2, p.\ 9]{Washington}. Moreover, if $q$ is a prime different from $p$ then $u_q(1-\zeta_{p^{k}})=0$. These calculations prove the claims about $S$. 
\end{Ex}

\begin{Ex}\label{Root of unity example 2} Let $S = \{\zeta_p\}_{p \in \PP}$. Once again, $\IntQ(S,\oZ)$ is trivial by Proposition \ref{index trivial conditions}. We will show that $S$ is pseudo-stationary with $\Br(E)=U_p$ with respect to $u_p$ for every prime $p$.

Given primes $p < q$, we have $\zeta_q - \zeta_p = \zeta_q(1 - \zeta_{pq}^{q-p})$, and $\zeta_{pq}^{q-p}$ is a primitive $pq$-th root of unity because $q-p$ is coprime to $pq$. So, for any prime $l$,
\begin{equation*}
u_l(\zeta_q - \zeta_p) = u_l(1 - \zeta_{pq}^{q-p}) = u_l(1 - \zeta_{pq}).
\end{equation*}
By \cite[Proposition 2.8]{Washington}, $1-\zeta_{pq}$ is a unit of $\Z[\zeta_{pq}]$. Hence,  $u_l(\zeta_q - \zeta_p) = 0$, and so for every prime $l$, $S$ is pseudo-stationary with respect to $u_l$ with breadth ideal $U_l$.
\end{Ex}

Note that in both Example \ref{Root of unity example} and Example \ref{Root of unity example 2}, one could use Lemma \ref{Extensions to Qbar lemma 2} instead of Proposition \ref{index trivial conditions} to conclude that $\IntQ(S,\oZ)$ is trivial.

It is not possible for a sequence $S = \{s_k\}_{k \in \N} \subseteq \oZ$ to be pseudo-divergent with respect to all primes. If that were the case, then for all $i \ne j$, the difference $s_i-s_j$ would have finite positive value at each integral prime, which is impossible. However, we can construct a sequence $S$ that is \textit{eventually} pseudo-divergent with respect to each prime.

\begin{Ex}\label{Pseudo-divergent example}
Let $\PP=\{p_1, p_2, \ldots\}$. Define $s_1 = p_1$, $s_2 = (p_1p_2)^{1/2}$, $s_3 = (p_1p_2p_3)^{1/3}$, and in general $s_k = (p_1 \cdots p_k)^{1/k}$ for each $k \in \N$. 
Then, for each $n \in \N$,
\begin{equation*}
u_{p_n}(s_k) = \begin{cases} 0, & k < n\\ \tfrac{1}{k}, & k \geq n.\end{cases}
\end{equation*}
Thus, for every prime $p$, the sequence $\{u_p(s_k)\}_{k \in \N}$ eventually strictly decreases to 0. Hence, $S$ is eventually pseudo-divergent with breadth ideal equal to $\oMp$  with respect to every prime $p$. As in the prior two examples, $\IntQ(S,\oZ)$ is trivial.
\end{Ex}

\section{Ramification Indices and Residue Field Degrees}\label{e and f section}

We first show that having bounds on both the ramification indices and residue field degrees of elements of $S$ is sufficient for $\IntQ(S,\oZ)$ to be nontrivial. Similar conditions appear in \cite{LopIntD} in a theorem which classifies the almost Dedekind domains $D$ with finite residue fields such that $\Int(D)$ is a Pr\"{u}fer domain.

\begin{Lem}\label{Bounded ram and res field}
Let $S \subseteq \oZ$. Assume that there exists $p \in \PP$ such that $S$ has both bounded ramification indices and bounded residue field degrees at $p$. That is, assume that there exist $e_0, f_0 \in \N$ such that for all $s \in S$ and every prime $P_s$ of $O_{\Q(s)}$ above $p$, we have $e(P_s|p) \leq e_0$ and $f(P_s|p) \leq f_0$. Then, both $\IntQ(S,\oZ)$ and $\IntQ(S,\overline{\Z_{(p)}})$ are nontrivial.
\end{Lem}
\begin{proof}
Fix $s \in S$, and let $P_s$ be a prime of $O_{\Q(s)}$ above $p$. Let $q=p^{f_0!}$. Then, $O_{\Q(s)}/P_s$ is a subfield of $\F_{q}$, so $X^q-X$ maps $O_{\Q(s)}$ into $P_s$ under evaluation for every $s\in S$. Since all ramification indices over $p$ are bounded above by $e_0$, it follows that $(X^q-X)^{e_0}$ sends $O_{\Q(s)}$ into $pO_{\Q(s)}$. Because the values of $q$ and $e_0$ are independent of $s$, this is true for all elements of $S$. Thus, the polynomial $((X^q-X)^{e_0})/p$ is in $\IntQ(S,\oZ) \subseteq \IntQ(S,\overline{\Z_{(p)}})$, and both rings are nontrivial.
\end{proof}
 
The hypothesis of Lemma \ref{Bounded ram and res field} is clearly met when $S$ is of bounded degree. However, there exist sets $S$ of unbounded degree for which Lemma \ref{Bounded ram and res field} can be applied.

\begin{Ex}\label{Qn}
Let $\Q^{(n)}=\Q(\mathcal A_n)$ be the compositum in $\oQ$ of all number fields of degree bounded by $n$ and let $O_{\Q^{(n)}}$ be the ring of integers of $\Q^{(n)}$. It is known that there exists a global bound on both ramification indexes and residue field degrees of valuations of $\Q^{(n)}$. That is, there exists $N\in\N$ such that if $u_p$ is a valuation of $\Q^{(n)}$ extending some $v_p$, $p\in\PP$, then $e(u_p|v_p)\leq N$ and $f(u_p|v_p)\leq N$ \cite[Proposition 4.5.3, p. 118]{BG}; we stress that $N$ is independent from $p$. So, any subset $S \subseteq O_{\Q^{(n)}}$ meets the conditions of Lemma \ref{Bounded ram and res field}, and hence $\IntQ(S,\oZ)$ is nontrivial. Furthermore, subsets of $O_{\Q^{(n)}}$ may have unbounded degree. For an explicit example, let $\PP=\{p_k\}_{k\in\N}$. For each $k\in\N$, let $s_k=\sum_{i=1}^k\sqrt{p_i}$, and take $S=\{s_k\}_{k\in\N}$. Then, $S \subseteq O_{\Q^{(2)}}$ and for each $k$, $[\Q(s_k):\Q] = 2^k$.
\end{Ex}

As mentioned earlier, the double-boundedness condition of Lemma \ref{Bounded ram and res field} was used in \cite{LopIntD} to classify the almost Dedekind domains $D$ with finite residue fields such that $\Int(D)$ is Pr\"{u}fer. We do not know whether this condition guarantees that $\IntQ(S,\oZ)$ is Pr\"{u}fer. However, we can prove that this is the case when $S \subseteq O_{\Q^{(n)}}$.

\begin{Prop}\label{Qn Prufer prop}
Let $n \in \N$. Define $\Q^{(n)}$ and $O_{\Q^{(n)}}$ as in Example \ref{Qn}, and let $S \subseteq O_{\Q^{(n)}}$. Then, $\IntQ(S,\oZ)$ is a Pr\"{u}fer domain.
\end{Prop}
\begin{proof}
For readability, let $K = \Q^{(n)}$ and $O = O_{\Q^{(n)}}$. As usual, we let $\Int(O) = \{f \in K[X] \mid f(O) \subseteq O\}$ and $\IntQ(O) = \Int(O) \cap \Q[X]$. Then,
\begin{equation*}
\IntQ(O) = \IntQ(O,\oZ) \subseteq \IntQ(S,\oZ).
\end{equation*}
Since overrings of Pr\"{u}fer domains are Pr\"{u}fer, it suffices to prove that $\IntQ(O)$ is Pr\"{u}fer. By the properties of ramification indexes and residue field degree mentioned in Example \ref{Qn}, it follows that $O$ is an almost Dedekind domain with finite residue fields satisfying the double-boundedness condition of \cite[Theorem 2.5]{LopIntD}. Hence, $\Int(O)$ is Pr\"{u}fer.  By \cite[Lemma 2.20]{PerDedekind}, $\Int(O)$ is integral over $\IntQ(O)$, so $\IntQ(O)$ is Pr\"ufer by \cite[Theorem 22.4]{Gilm}. 
\end{proof}

\begin{Rem}\label{Qn remark}
For each $n \in \N$, we have $\IntQ(O_{\Q^{(n)}})\subseteq\IntQ(\mathcal{A}_n)$. As in \eqref{An inclusion chain}, we have a descending chain of Pr\"{u}fer domains:
\begin{equation*}
\ldots\subseteq\IntQ(O_{\Q^{(n+1)}})\subseteq\IntQ(O_{\Q^{(n)}})\subseteq\ldots\subseteq\IntQ(O_{\Q^{(1)}})=\Int(\Z).
\end{equation*}
It is not known whether each containment in this train is strict. Regardless, as was shown in the Introduction for the rings $\IntQ(\mathcal{A}_n), n\in\N$, we have $\bigcap_{n\in\N}\IntQ(O_{\Q^{(n)}})=\IntQ(\oZ)=\Z[X]$.

\end{Rem}

While bounds on both ramification indices and residue field degrees are sufficient to conclude that $\IntQ(S,\oZ)$ is nontrivial by Lemma \ref{Bounded ram and res field}, neither condition is necessary for nontriviality. The remainder of this section is devoted to examples that illustrate this. We also consider how to interpret the presence of either unbounded residue field degrees (respectively, ramification degrees) or in terms of pseudo-stationary (respectively, pseudo-divergent) sequences in suitable images of $S$ in $\oZp$.

\begin{Ex}\label{Unbounded ram lemma}
Let $S \subseteq \oZ$ and $p \in \PP$. Let $I$ be the set of all ramification indices $e(P_s|p)$, where $s$ runs through $S$ and $P_s$ runs through the prime ideals of $O_{\Q(s)}$ above $p$. We give two examples to show that if $I$ is unbounded, then $\IntQ(S,\oZ)$ may or may not be trivial.

First, let $S$ be as in Example \ref{nth roots of p example}. The associated set of ramification indices is $I = \{2^k \mid k \in \N\}$, which is unbounded. However, $X^2/p \in \IntQ(S,\oZ)$, so the ring of integral-valued polynomials is nontrivial. By contrast, if $S$ is the set defined in Example \ref{Pseudo-divergent example}, then $S$ still exhibits unbounded ramification indices, but $\IntQ(S,\oZ)$ is trivial.
\end{Ex}

As for the situation of ramification indices, we can show that the bound of the residue field degrees is not necessary for the $\IntQ$-ring to be nontrivial.

\begin{Ex}
    For a fixed prime $p\in\PP$, we consider the $m^{\text{th}}$ roots of unity in $\oZ$ with $(p,m)=1$; we recall that $p$ is not ramified in $\Q(\zeta_m)$ \cite[Corollary of Theorem 26]{Marcus}. The set of residue field degrees of the set of prime ideals of the finite extensions generated by the elements of the set $S=\{p\cdot\zeta_m\mid m\in\N, (p,m)=1\}$ is unbounded but $u_p(p\cdot\zeta_m)>0$ for each of the relevant $m$'s.  Clearly, $\frac{X}{p}\in\IntQ(S,\oZ)$. The important condition here is that the residue set $S/\oP$, for each prime ideal $\oP\subset\oZ$ above $p$, is finite.
\end{Ex}

\begin{Not}
For $p \in \PP$, we denote by $\mcP_p$ the set of prime ideals $\oP$ of $\oZ$ lying above $p$.
\end{Not}

\begin{Lem}\label{Unbounded res field deg lem}
Let  $S\subseteq\overline{\Z_{(p)}}$ for some  $p\in\PP$. Suppose there exists $n\in\N$ such that either one of the following conditions is satisfied:
\begin{itemize}
    \item[(i)] there exists some $\oP\in\mathcal{P}_p$ such that $\# S/\oP>n$.
    \item[(ii)] there exist some $\oP\in\mathcal{P}_p$ and $s\in S$ such that $0<v_{\oP}(s)<\frac{1}{n}$.
\end{itemize}
Then, given $f\in\IntQ(S,\overline{\Z_{(p)}})\setminus\Z_{(p)}[X]$, we have $\deg(f)>n$. In particular, if for every $n\in\N$ either one of these conditions hold, then $\IntQ(S,\overline{\Z_{(p)}})=\Z_{(p)}[X]$.
\end{Lem}
\begin{proof}
Let $f\in\IntQ(S,\overline{\Z_{(p)}})\setminus\Z_{(p)}[X]$ be of degree $d\leq n$; without loss of generality, we may assume that $f=\frac{g}{p}$ for some monic $g\in\Z_{(p)}[X]$ (see the arguments of the proof of Proposition \ref{Regular basis} for the monic assumption; we may multiply $f$ by a suitable power of $p$ in order to get $p$ at the denominator). 
\vskip0.2cm
(i)   By assumption, there exists a prime ideal $\oP\in\mathcal P_p$ such that $\#S/\oP>n$, that is, there exist $s_0,\ldots,s_n\in S$ such that $u_{\oP}(s_i-s_j)=0$ for each $0\leq i<j\leq n$. In particular, $\overline{s_i}=s_i\pmod{\oP}$ are $n+1$ distinct elements of $\overline{\F_p}$. Then $g(s_i)\equiv0\pmod \oP$, for $i=0,\ldots,n$, that is, $\overline{g}(\overline{s_i})=0$ in $\overline{\F_p}$, where $\overline{g}\in\F_p[X]$ is the reduction modulo $p$ of $g$. It follows that  $\overline{g}\in\F_p[X]$ is a polynomial of degree $\leq n$ which has $n+1$ distinct roots in $\overline{\F_p}$, a contradiction.
   
\vskip0.3cm
(ii) By assumption, there exists a prime ideal $\oP\in\mathcal{P}_p$ and $s\in S$ such that $v_{\oP}(s)<\frac{1}{n}$. Let  $g(X)=\prod_{i=1}^d(X-\alpha_i)$ over $\oQ$, where each $\alpha_i \in \An$. Fix $i\in\{1,\ldots,n\}$. If $v_{\oP}(\alpha_i)$ is not $0$ then it is equal to $\frac{a}{n}>0$, for some $a\in\N$ (not necessarily coprime with $n$) since $e(\oP\cap O_{\Q(\alpha_i)}\mid p)\leq [\Q(\alpha_i):\Q]\leq n$. If this second case occurs, then $0<v_{\oP}(s)<\frac{1}{n}\leq v_{\oP}(\alpha_i)$. In either case we have:
$$v_{\oP}(s-\alpha_i)=\min\{v_{\oP}(\alpha_i),v_{\oP}(s)\}=
\begin{cases}
0,&\text{ if }v_{\oP}(\alpha_i)=0\\
v_{\oP}(s),&\text{ otherwise}.
\end{cases}$$
Hence, $v_{\oP}(g(s))\leq \sum_{i=1}^n v_{\oP}(s)\leq nv_{\oP}(s)<1=v_{\oP}(p)$ which contradicts the fact that $f(s)\in\overline{\Z_{(p)}}\subset\oZ_{\oP}$.
\end{proof}

\begin{Ex}\label{Gilmer example}
For an explicit example of a subset $S$ of $\oZ$ for which $S/\oP$ is finite for each maximal ideal $\oP\in \mcP_p$ for some prime $p\in\Z$ but $\{\#S/\oP\mid\oP\in\mathcal{P}_p\}$ is unbounded, let $S$ be equal to the almost Dedekind domain $D$ with finite residue fields of \cite[Example 14]{GilmEx} (see also \cite[Example VI.4.18]{CaCh}). In that example, $D$ is the integral closure of $\Z_{(p)}$ in a suitable infinite algebraic extension of $\Q$ constructed as the union of a tower of finite algebraic extensions, so by Lemma \ref{Unbounded res field deg lem} (see also \cite[Lemma VI.4.2]{CaCh}) applied to that example shows that $\IntQ(D,\oZ)=\IntQ(D)=\Z_{(p)}[X]$.  
\end{Ex}

Note that in Example \ref{Gilmer example} there are neither pseudo-stationary sequence nor pseudo-divergent sequences in $D$ with respect to every possible extension of $v_p$ to $\oQ$ because $D$ is locally a DVR with finite residue field. Nonetheless, $\IntQ(D)$ is trivial. The reason for this is that if we work with all of the $\Q$-embeddings of $\oQ$ into $\oQp$, and consider the union of all of the images of $D$ under these embeddings, then we obtain the same $\IntQ$-ring (Lemma \ref{transfer to oZp}) and this union contains a pseudo-stationary sequence (see Proposition \ref{Unbounded res field deg prop}).

\begin{Def} 
For each $\oP \in \mcP_p$, let $u_{\oP}$ be the valuation associated to $\oZ_{\oP}$. By \cite[Chapter 2, \S 8, Theorem 8.1]{Neu}, for each $\oP\in\mathcal{P}_p$, the extension $u_{\oP}$ of $v_p$ to $\oQ$ is equal to $v_p\circ\tau_{\oP}$, for some $\Q$-embedding $\tau_{\oP}:\oQ\hookrightarrow\oQp$, where $v_p$ is the unique extension to $\oQp$ of the $p$-adic valuation. Given $S \subseteq \oZ$, we define the following subset of $\oZp$:
\begin{equation*}
\mathcal{S}_p=\bigcup_{\oP\in\mathcal{P}_p}\tau_{\oP}(S).
\end{equation*}
\end{Def}

The following lemma is analogous to Proposition \ref{IntAnpcompletion}.

\begin{Lem}\label{transfer to oZp}
Let $p\in\PP$, $S \subseteq \overline{\Z_{(p)}}$ and $\oP\in\mathcal{P}_p$. Then, we have
$$\IntQ(S,\oZ_{\oP})=\IntQ(\tau_{\oP}(S),\oZp).$$ 
In particular,
$$\IntQ(S,\overline{\Z_{(p)}})=\IntQ(\mathcal{S}_p,\oZp).$$ 
\end{Lem}
\begin{proof}
$(\subseteq)$ Let $f\in\IntQ(S,\oZ_{\oP})$ and $\tau_{\oP}(s)\in \tau_{\oP}(S)$, for some $s\in S$. Then $f(s)\in\oZ_{\oP}$, so $u_{\oP}(f(s))\geq0$. Hence, $$v_p\circ \tau_{\oP}(f(s))=v_p(f(\tau_{\oP}(s)))\geq0$$
so that $f(\tau_{\oP}(s))\in\oZp$. It follows that $f\in\IntQ(\tau_{\oP}(S),\oZp)$.

$(\supseteq)$ Let $f\in\IntQ(\tau_{\oP}(S),\oZp)$ and $s\in S$. By definition we have $f(\tau_{\oP}(s))\in\oZp$. Hence, applying $\tau_{\oP}^{-1}$ to the previous relation we get $f(s)\in\oZ_{\oP}$, and so $f\in\IntQ(S,\oZ_{\oP})$.

The final claim follows from the facts that $\overline{\Z_{(p)}} = \bigcap_{\oP\in\mcP_p} \oZ_{\oP}$ and $\bigcap_{\oP\in\mcP_p}\IntQ(\tau_{\oP}(S),\oZp)=\IntQ(\mathcal{S}_p,\oZp)$.
\end{proof}

\begin{Rem}\label{S_p remark}
Recall the set $\Sigma_p(S)$ that was introduced in Definition \ref{Sigma_p}. One can show that $\Sigma_p(S)$ is equal to the $\Gal(\oQp/\Q_p)$-closure of $\mathcal{S}_p$. Consequently, many of the results in Section \ref{Local section} can be stated in terms of $\mathcal{S}_p$ rather than $\Sigma_p(S)$. In particular, the following equivalences hold:
$$\IntQ(S,\overline{\Z_{(p)}})\text{ is nontrivial }\stackrel{\text{Prop} 3.9}{\Longleftrightarrow}\IntQ(\mathcal{S}_p,\oZp)\text{ is nontrivial }\stackrel{\text{Prop} 3.10}{\Longleftrightarrow}\Int_{\oQp}(\mathcal{S}_p,\oZp)\text{ is nontrivial }$$
The proof of each equivalence is identical to that of the cited statement. One can then apply Theorem \ref{Int(S,V) nontrivial V1} to the ring $\Int_{\oQp}(\mathcal{S}_p,\oZp)$ to establish whether it is nontrivial or not.
\end{Rem}

\begin{Prop}\label{Unbounded res field deg prop}
Let $S\subseteq\overline{\Z_{(p)}}$ and let $\oMp$ be the maximal ideal of $\oZp$. The following are equivalent.
\begin{enumerate}[(1)]
\item $\{\# S/\oP\mid\oP\in \mathcal{P}_p\}$ is unbounded.
\item $\mathcal{S}_p/\oMp$ is infinite.
\item $\mathcal{S}_p$ contains a pseudo-stationary sequence $E$ (with respect to $v_p$, the unique valuation of $\oZp$) such that $\Br(E) = \oZp$.
\end{enumerate}
Furthermore, if any of these conditions holds, then $\IntQ(S,\overline{\Z_{(p)}})$ is trivial.
\end{Prop}
\begin{proof} 

(1) $\Rightarrow$ (2). Suppose that $\{\# S/\oP\mid\oP\in \mathcal{P}_p\}$ is unbounded, but $\# \mathcal{S}_p/\oMp = n$ for some $n \in \N$. By assumption, there exists some $\oP \in\mathcal P_p$ such that $\#S/\oP > n$. Thus, there exist $s_0,\ldots,s_n\in S$ such that $u_{\oP}(s_i-s_j)=0$ for each $0\leq i<j\leq n$. Applying the $\Q$-embedding $\tau_{\oP}$ to each $s_i$, we obtain $n+1$ elements $\tau_{\oP}(s_0), \ldots, \tau_{\oP}(s_n) \in \mathcal{S}_p$ such that 
\begin{equation*}
v_p(\tau_{\oP}(s_i)-\tau_{\oP}(s_j)) = v_p\circ\tau_{\oP}(s_i-s_j) = u_{\oP}(s_i-s_j)=0
\end{equation*}
for each $0\leq i<j\leq n$. This is impossible, since some two of the $\tau_{\oP}(s_i)$ must lie in the same residue class modulo $\oMp$.  We conclude that $\mathcal{S}_p/\oMp$ is infinite.

(2) $\Rightarrow$ (1). Assume that $\mathcal{S}_p/\oMp$ is infinite. Clearly, we have the identities
$$\frac{\phantom{x}\mathcal{S}_p\phantom{x}}{\oMp}=\frac{\bigcup_{\oP\in\mathcal{P}_p}\tau_{\oP}(S)}{\oMp}=\bigcup_{\oP\in\mathcal{P}_p}\frac{\tau_{\oP}(S)}{\oMp}.$$
Moreover, for each $\oP\in\mathcal{P}_p$, $\tau_{\oP}$ establishes a bijection between $S/\oP$ and $\tau_{\oP}(S)/\oMp$. If $\tau_{\oP}(S)/\oMp$ is infinite for some $\oP\in\mathcal{P}_p$, then considering the pullback via $\tau_{\oP}^{-1}$ we have that $S/\oP$ is infinite. If instead $\tau_{\oP}(S)/\oMp$ is finite for each $\oP\in\mathcal{P}_p$, since $\mathcal{S}_p/\oMp$ is infinite, this means that there exists a sequence $\{\oP_n\}_{n\in\N}\subseteq\mathcal{P}_p$ such that $\#\tau_{\oP_n}(S)/\oMp=d_n<\infty$ is diverging to infinity. This amounts to say via the previous one-to-one correspondence that $\#S/\oP_n=d_n$ is diverging to infinity, which means that $\{\# S/\oP\mid\oP\in \mathcal{P}_p\}$ is unbounded.

(2) $\Leftrightarrow$ (3) If $\mathcal{S}_p/\oMp$ is infinite, then we can choose a sequence $E = \{t_i\}_{i \in \N} \subseteq \mathcal{S}_p$ such that each $t_i$ lies in a distinct residue class modulo $\oMp$. We then have $v_p(t_i - t_j) = 0$ whenever $i \ne j$, so $E$ is pseudo-stationary and $\Br(E) = \oZp$. Conversely, if $\mathcal{S}_p/\oMp$ is finite, then no such pseudo-stationary sequence can exist in $S$.

For the last claim, apply Theorem \ref{Int(S,V) nontrivial V1} and the equivalences noted in Remark \ref{S_p remark}.
\end{proof}

We conjecture that there is an analogous result linking unbounded ramification indices of $S$ at $p$ to pseudo-divergent sequences in $\mathcal{S}_p$. However, determining the correct conditions for such a relation to hold in general has proven to be difficult. Nevertheless, we can provide one example showing that such a connection is possible, at least in the case where $S$ is a ring.

\begin{Ex}\label{Chabert example}
Following \cite[Example 6.2]{ChabEx} (see also \cite[Example VI.4.17]{CaCh}), one may construct an almost Dedekind domain $D$ with the following properties:
\begin{itemize}
\item $D$ has finite residue fields.
\item $D$ is the integral closure of $\Z_{(p)}$ in a (suitably chosen) infinite algebraic extension $K$ of $\Q$.
\item The set $\{e(\oP\cap D\mid p)\mid \oP\in\mathcal{P}_p\}$ is unbounded.
\end{itemize}
For this domain $D$, every valuation overring is a DVR, and every residue field of $D$ is equal to $\F_p$. These conditions guarantee that the ramification indices of $D$ are bounded. Furthermore, $D$ cannot contain any pseudo-divergent sequence with respect to any of its maximal prime ideals $\oP\cap D$, because $D_{\oP\cap D}$ is a DVR. We will show that $\mathcal{D}_p':=\bigcup_{\oP\in\mathcal{P}_p}\tau_{\oP}(D_{\oP\cap D})\subseteq\oZp$ contains a pseudo-divergent sequence and by an argument similar to Lemma \ref{transfer to oZp}, the ring $\IntQ(D,\overline{\Z_{(p)}})$ must equal $\Z_{(p)}[X]$.
\end{Ex}

\begin{Prop}\label{triviality unbounded ramification}
Let $D\subseteq\overline{\Z_{(p)}}$ be a ring containing $\Z_{(p)}$ and define $\mathcal{D}_p'$ as above. The following are equivalent.
\begin{enumerate}[(1)]
\item $\{e(P\mid p)\mid p\in P\subset D\}$ is unbounded.
\item For each $n\in\N$, there exist $P_n\subset D$, with $p\in P_n$ and $s_n\in D_{P_n}$ such that $0<v_{P_n}(s_n)<\frac{1}{n}$.
\item There exists a pseudo-divergent sequence $E\subset\mathcal{D}_p'$ with $\Br(E)=\overline{M_p}$.
\end{enumerate}
Moreover, if any one of these conditions holds, then $\IntQ(D)=\Z_{(p)}[X]$, and thus is trivial.
\end{Prop}
\begin{proof}
(1) $\Rightarrow$ (2) If for some prime ideal $P$ of $D$ containing $p$ we have $e(P\mid p)=\infty$, then it is easy to see that for each $n\in\N$, there exists $s_n\in D_P$ with $0<v_P(s_n)<\frac{1}{n}$. Suppose instead  that exists a sequence $\{P_n\}_{n\in\N}$ of prime ideals of $D$ containing $p$ such that $e(P_n\mid p)=e_n>n$ for each $n\in\N$. We claim that there exists $s_n\in D_{P_n}$ such that $v_{P_n}(s_n)=\frac{1}{e_n}<\frac{1}{n}$  which is precisely condition (2). Indeed, for each $n\in\N$ there exists $\alpha_n\in D$ such that $v_{P_n}(\alpha_n)=\frac{a_n}{e_n}$, for some $a_n\in \Z$ coprime with $e_n$. Now, if $aa_n+be_n=1$ for some $a,b\in\Z$, then $s_n=\alpha_n^a\cdot p^b$ is an element of $D_{P_n}$ with the desired value. 


(2) $\Rightarrow$ (3) Suppose (2) holds. Then for each $n\in\N$ we have $0<v_{P_n}(s_n)=v_p(\tau_{P_n}(s_n))=v_p(\tilde{s}_n)<\frac{1}{n}$, where $\tilde{s}_n\in\mathcal{D}_p'$. We remark that $P_n$ could be the same prime ideal for infinitely many $n$'s; this corresponds to the case $e(P\mid p)=\infty$. Suppose that for some $n\in\N$, we have $\tilde{s}_1,\ldots,\tilde{s}_n$ so that 
$0<v_p(\tilde{s}_{i+1})<v_p(\tilde{s}_i)$ for all $i=1,\ldots,n-1$. If we take $m\in\N$ so that $0<\frac{1}{m}\leq v_p(\tilde{s}_n)$, then by assumption there exists $\tilde{s}_{n+1}\in\mathcal{D}_p'$ such that $0<v_p(\tilde{s}_{n+1})<\frac{1}{m}\leq v_p(\tilde{s}_n)$. In this way we can extract a subsequence $E$ from $\{\tilde{s}_n\}_{n\in\N}\subset\mathcal{D}_p'$ which is pseudo-divergent (i.e., $0<v_p(\tilde{s}_{n+1})<v_p(\tilde{s}_{n})$ for each $n\in\N$) and has $\Br(E)=\overline{M_p}$.

(3) $\Rightarrow$ (1) Assume that (3) holds. Then, there exists a pseudo-divergent sequence $E=\{\tilde{s}_n\}_{n\in\N}\subset\mathcal{D}_p'$   with $\Br(E)=\overline{M_p}$. Hence, $v_p(\tilde{s}_n)$ strictly decreases to $0$ and, up to taking a subsequence, we may assume that $0<v_p(\tilde{s}_n)<\frac{1}{n}$. Since $\tilde{s}_n=\tau_{P_n}(s_n)$ for some $P_n\subset D$, $p\in D$ and $s_n\in D_{P_n}$, we have $0<v_{P_n}(s_n)=v_p(\tau_{P_n}(s_n))<\frac{1}{n}$ which implies that $n<e(P_n\mid p)$. Thus $\{e(P_n\mid p)\}_{n\in\N}$ is unbounded. Note that, as in the proof of (2) $\Rightarrow$ (3), if $P_n$ is the same prime ideal $P$ for infinitely many $n$'s, then  $e(P\mid p)=\infty$.

The last claim follows from Lemma \ref{Unbounded res field deg lem}.
\end{proof}

We end this section by showing that when $S$ is an integrally closed subring of $\overline{\Z_{(p)}}$ containing $\Z_{(p)}$, then the conditions of Lemma \ref{Bounded ram and res field} are both necessary and sufficient for $\IntQ(S,\oZ)$ to be nontrivial.

\begin{Not}
For an algebraic extension $K$ of $\Q$ and a prime $p\in\PP$, let $D$ be the integral closure of $\Z_{(p)}$ in $K$ (note that $D=O_{K,(p)}$, the localization of the integral closure $O_K$ of $\Z$ in $K$ with respect to $p$). We set
$$E_{D,p}:=\{e(P|p) \mid P\subset D\} \quad \text{ and } \quad F_{D,p}:=\{f(P|p) \mid P\subset D\}.$$
\end{Not}

\begin{Lem}\label{bounded degree Galois closure}
Let $K$ be an infinite algebraic extension of $\Q$ and let $K'$ be the Galois closure of $K$ over $\Q$. Let $p\in\PP$ and let $D$ and $D'$ be the integral closures of $\Z_{(p)}$ in $K$ and $K'$, respectively. If $E_{D,p}$ and $F_{D,p}$ are bounded, then $E_{D',p}$ and $F_{D',p}$ are bounded.
\end{Lem}
\begin{proof}
Let $\sigma\in G=\Gal(\oQ/\Q)$. Note that $E_{\sigma(D),p}=E_{D,p}$ and $F_{\sigma(D),p}=F_{D,p}$. Since $\Q_p$ has only finitely many extensions of bounded degree, there exists a finite extension $L$ of $\Q_p$ such that, for every $\sigma\in G$, $L$ contains the completion of $\sigma(K)$ with respect to any prime ideal of $\sigma(D)$. In particular, there are only finitely many such completions.

Note that $K'$ is equal to the compositum of the fields $\sigma(K), \sigma\in G$. Let $Q$ be a prime ideal of $D'$. We claim that the completion $\widehat{K'}$ of $K'$ with respect to $Q$ is equal to the compositum of the completions $\widehat{\sigma(K)}$ of $\sigma(K)$ with respect to $Q\cap\sigma(D)$ (which by the above are finite in number). Indeed, $\widehat{K'}$ contains $\widehat{\sigma(K)}$ for every $\sigma\in G$ and so $\widehat{K'}$ contains their compositum; conversely, since $\widehat{K'}$ is the completion of $K'$, it is contained in the compositum of the (finitely many) fields $\widehat{\sigma(K)},\sigma\in G$, because the latter is complete and contains $K'$.

Now, if $L'$ is the Galois closure of $L$ over $\Q_p$, then it is clear that $\widehat{K'}$ is contained in $L'$. Since this holds for every prime ideal $Q$ of $D'$ (and $L'$ is independent from the choice of the prime ideal $Q$), it follows that $E_{D',p}$ and $F_{D',p}$ are bounded.
\end{proof}

\begin{Thm}\label{nontrivial Prufer for rings}
Let $D$ be an integrally closed subring of $\overline{\Z_{(p)}}$ containing $\Z_{(p)}$. The following conditions are equivalent.
\begin{enumerate}[(1)]
\item $\IntQ(D)$ is Pr\"ufer.
\item $\IntQ(D)$ is nontrivial (i.e., $\Z_{(p)}[X] \subsetneq \IntQ(D)$).
\item The sets $F_{D,p}$ and $E_{D,p}$ are bounded.
\end{enumerate}
\end{Thm}
\begin{proof}
Clearly, (1) implies (2).

(2) $\Rightarrow$ (3) If $F_p$ is unbounded then by Proposition \ref{Unbounded res field deg prop} and Remark \ref{S_p remark}, $\IntQ(D)$ is trivial. Similarly, if $E_p$ is unbounded the same conclusion holds by Proposition \ref{triviality unbounded ramification}.

(3) $\Rightarrow$ (1)  Let $K'$ be the Galois closure of $K$ and let $D'$ be the integral closure of $D$ in $K'$. By Lemma \ref{bounded degree Galois closure}, $E_{D',p}$ and $F_{D',p}$ are bounded. Therefore, $\Int_{K'}(D')$ is Pr\"ufer by \cite[Theorem 2.5]{LopIntD}. We consider now the set $G(D)=\bigcup_{\sigma\in G}\sigma(D)\subset D'$. Since $\Int_{K'}(D')\subseteq\Int_{K'}(G(D),D')$, the latter ring is Pr\"ufer, too. By \cite[Lemma 2.20]{PerDedekind}, $\Int_{K'}(G(D),D')$ is the integral closure in $K'(X)$ of $\IntQ(D)=\IntQ(G(D),D')$.  It follows by \cite[Theorem 22.4]{Gilm} that $\IntQ(D)=\Int_{K'}(G(D),D')\cap\Q(X)$ is Pr\"ufer.
\end{proof}

\begin{Cor}
  Let $D\subseteq\oZ$ be an integrally closed subring.
  \begin{enumerate}[(1)]
      \item $\IntQ(D)$ is nontrivial if and only if there exists some $p\in\PP$ such that both $E_{D,p}$ and $E_{F,p}$ are bounded.
      \item $\IntQ(D)$ is Pr\"ufer if and only if for each $p\in\PP$ the sets $E_{D,p}$ and $E_{D,p}$ are bounded.
  \end{enumerate}
\end{Cor}
\begin{proof}
    (1) This follows by Proposition \ref{1 iff 2} and Theorem \ref{nontrivial Prufer for rings}.

    (2) It is easy to show that $\IntQ(D)$ is Pr\"ufer if and only if for each $p\in\PP$, the localization $\IntQ(D)_{(p)}$ is Pr\"ufer. By Proposition \ref{localization IntQSZ V1}, the latter ring is equal to $\IntQ(D,\overline{\Z_{(p)}})$, and since $\overline{\Z_{(p)}}\cap K=D_{(p)}$, we have $\IntQ(D,\overline{\Z_{(p)}})=\IntQ(D,D_{(p)})$. Finally, $\IntQ(D,D_{(p)})=\IntQ(D_{(p)})$ by \cite[Corollary I.2.6]{CaCh}. The conclusion then follows by Theorem \ref{nontrivial Prufer for rings}.   
    \end{proof}

\vskip1cm

\section{Polynomial closure of subsets of algebraic integers}\label{Poly closure section}

Recall that by Example \ref{2Z example}, when $S=2\oZ = \{2\alpha \mid \alpha \in \oZ\}$, the ring $\IntQ(S,\oZ)$ is nontrivial because it contains $X/2$. This example admits a rather strong generalization.

\begin{Def}\label{S(f,d) sets}
Let $f \in \Z[X]$ be nonconstant and monic. Let $d \in \Z, d \geq 1$. For each $\alpha \in \oZ$, let $\mathcal{Z}(f(X)-d\alpha) \subseteq \oZ$ be the set of all the roots of $f(X)-d\alpha$. Define $\msS(f,d)$ to be the following subset of $\oZ$:
\begin{equation*}
\msS(f,d) := \bigcup_{\alpha \in \oZ} \mcZ(f(X)-d\alpha).
\end{equation*}
Note that if $d=1$, then $\msS(f,1) = \oZ$.
\end{Def}

With this notation, $2\oZ = \msS(X,2)$. The sets $\msS(f,d)$ provide many examples of subsets of $\oZ$ for which the associated integral-valued polynomial ring is nontrivial.

\begin{Ex}\label{unbounded nontrivial examples}
Let $f \in \Z[X]$ be nonconstant and monic, and let $d \geq 2$. Then, $\msS(f,d)$ has unbounded degree, and $\IntQ(\msS(f,d), \oZ)$ contains $f(X)/d$, hence is nontrivial. 
\end{Ex}

There is more to say about the relationship between the sets $\msS(f,d)$ and $\IntQ(S,\oZ)$.

\begin{Def}\label{Polynomial closure}
Given a subset $S \subseteq \oZ$, the \textit{polynomial closure} of $S$ in $\oZ$ is the largest subset $T \subseteq \oZ$ containing $S$ such that $\IntQ(S,\oZ) = \IntQ(T,\oZ)$. We say that $S$ is \emph{polynomially closed in $\oZ$} if $S=T$. In the literature \cite{ChabPolCloVal, PS4}, the polynomial closure of $S$ is often denoted by $\overline{S}$. We will not employ this notation in order to avoid confusion with integral closures or algebraic closures.
\end{Def}

In Theorem \ref{polynomially closed sets} below, we prove that each polynomially closed subset of $\oZ$ is equal to an intersection of $\msS(f,d)$ sets. In particular, this means that if $\IntQ(S,\oZ)$ is nontrivial, then $S \subseteq \msS(f,d)$ for some $f$ and some $d \geq 2$.

We recall now the definitions of characteristic ideals and regular bases, as discussed in \cite[Section II.1]{CaCh}. For each integer $n \geq 0$, let $\mfI_n(S)$ be the $n^\text{th}$ characteristic ideal of $\IntQ(S,\oZ)$ of degree $n$, which consists of 0 and all the leading coefficients of polynomials in $\IntQ(S,\oZ)$ of degree $n$. For each $n \geq 0$, $\mfI_n(S)$ is a principal fractional ideal of $\Z$ that contains $\Z$. Let $R$ be a ring such that $\Z[X] \subseteq R \subseteq \Q[X]$. Then, a regular basis for $R$ is a $\Z$-module basis $\{b_n\}_{n \geq 0}$ for $R$ such that each $b_n$ has degree $n$. This happens precisely when $\mfI_n(S)$ is principal, generated by the leading coefficient of $b_n$, for each $n\in\N$ \cite[Theorem II.1.4]{CaCh}.

\begin{Prop}\label{Regular basis}
Let $S \subseteq \oZ$. Then, $\IntQ(S,\oZ)$ has a regular basis of the form $\{f_n(X)/d_n\}_{n \geq 0}$ such that for each $n \geq 0$, $f_n \in \Z[X]$ is monic and $d_n$ is a positive integer.
\end{Prop}
\begin{proof}
Note that in the case $\IntQ(S,\oZ)$ is trivial, a regular basis is given by $\{X^n\}_{n\in\N}$.

For each $n \geq 0$, the characteristic ideal $\mfI_n(S)$ is a principal fractional ideal of $\Z$. Fix $n$, and assume that $\mfI_n(S)$ is generated by $\tfrac{a}{b} \in \Q$, where $a$ and $b$ are coprime integers and $b > 0$. Let $x, y \in \Z$ such that $ax+by=1$. Since $\Z \subseteq \mfI_n(S)$, we have $\tfrac{1}{b} = x(\tfrac{a}{b}) + y \in \mfI_n(S)$. Thus, we may assume without loss of generality that each $\mfI_n(S)$ is generated by $\tfrac{1}{d_n}$, where $d_n \in \Z$ and $d_n \geq 1$.

Now, by \cite[Proposition II.1.4]{CaCh}, $\IntQ(S,\oZ)$ has a regular basis $\{g_n\}_{n \geq 0} \subseteq \Q[X]$ such that for each $n$, $g_n$ has degree $n$ and the leading coefficient of $g_n$ generates $\mfI_n(S)$. Moreover, by \cite[Proposition II.1.7]{CaCh}, each coefficient of $g_n$ is an element of $\mfI_n(S)$. In light of the previous paragraph, we see that there are integers $c_0, \ldots, c_{n-1}$ such that
\begin{equation*}
g_n(X) = \tfrac{1}{d_n}X^n + \sum_{i=0}^{n-1} \tfrac{c_i}{d_n}X^i.
\end{equation*}
Taking $f_n(X) = X^n + \sum_{i=0}^{n-1} c_iX^i$, we achieve the desired regular basis.
\end{proof}

\begin{Thm}\label{polynomially closed sets}
Let $S \subseteq \oZ$. Let $\{f_n(X)/d_n\}_{n \geq 0}$ be a regular basis for $\IntQ(S,\oZ)$ such that each $f_n \in \Z[X]$ is monic and each $d_n$ is a positive integer. Let $\mathscr{F} = \{f_n(X)/d_n \mid d_n \geq 2\}$. 
\begin{enumerate}[(1)]
\item $\IntQ(S,\oZ)$ is nontrivial if and only if $\mathscr{F} \ne \varnothing$.
\item The polynomial closure of $S$ in $\oZ$ is $\displaystyle\bigcap_{f(X)/d \in \mathscr{F}} \msS(f,d)$.
\end{enumerate}
\end{Thm}
\begin{proof}
Part (1) is clear. For (2), let $T$ be the polynomial closure of $S$ in $\oZ$. If $\mathscr{F}$ is empty, then $T = \oZ$ because $\IntQ(S,\oZ)=\Z[X]=\IntQ(\oZ,\oZ)$, and an intersection of an empty collection of subsets of $\oZ$ is also equal to $\oZ$. So, assume that $\mathscr{F} \ne \varnothing$. Let $\alpha \in T$, and let $f(X)/d \in \mathscr{F}$. Then, $f(X)/d \in \IntQ(S,\oZ) = \IntQ(T,\oZ)$, so $f(\alpha)/d \in \oZ$ and $\alpha \in \msS(f,d)$. Thus,
\begin{equation*}
T \subseteq \bigcap_{f(X)/d \in \mathscr{F}} \msS(f,d).
\end{equation*}
Conversely, let $\beta \in \bigcap_{f(X)/d \in \mathscr{F}} \msS(f,d)$, let $g \in \IntQ(S,\oZ)$, and let $m = \deg g$. Since $\{f_n(X)/d_n\}_{n \geq 0}$ is a regular basis for $\IntQ(S,\oZ)$, there exist $c_0, \ldots, c_m \in \Z$ such that
\begin{equation*}
g(X) = \sum_{i=0}^m \dfrac{c_if_i(X)}{d_i}.
\end{equation*}
For each $0 \leq i \leq m$, if $d_i = 1$ then clearly $f_i(\beta)/d_i \in \oZ$. If $d_i \geq 2$, then $f_i(X)/d_i \in \mathscr{F}$, so $\beta \in \msS(f_i,d_i)$, which means that $f_i(\beta)/d_i \in \oZ$. It follows that $g(\beta) \in \oZ$. Thus, $\beta \in T$, and therefore $T = \bigcap_{f(X)/d \in \mathscr{F}} \msS(f,d)$.
\end{proof}

We close the paper by using a theorem of McQuillan to prove that $\mathcal{A}_1 = \Z$ is polynomially closed in $\oZ$. It is an open problem to determine whether $\mathcal{A}_n$ is polynomially closed in $\oZ$ when $n \geq 2$, although we suspect that this is the case.

\begin{Thm}\cite[Theorem]{McQ Split}\label{McQ Theorem}
Let $\alpha\in\oZ$ and let $\Int(\Z)(\alpha)=\{f(\alpha)\mid f\in\Int(\Z)\}$. Then,
$$\Int(\Z)(\alpha)=\bigcap_{P\in\mathcal{S}_1}O_{\Q(\alpha),P}$$
where the intersection is over the family $\mathcal{S}_1$ of prime ideals $P$ of $O_{\Q(\alpha)}$ which are totally split over $\Z$, that is $e(P\mid P\cap\Z)f(P\mid P\cap\Z)=1$. 
\end{Thm}

\begin{Prop}
$\Z$ is polynomially closed in $\oZ$.
\end{Prop}
\begin{proof}
Given $\alpha\in\oZ\setminus\Z$, define $\Int(\Z)(\alpha)$ as in Theorem \ref{McQ Theorem}. Since not all the prime ideals of $O_{\Q(\alpha)}$ are totally split over $\Z$, by Theorem \ref{McQ Theorem} $\Int(\Z)(\alpha)$ is a proper overring of $O_{\Q(\alpha)}$. Thus, there exists $f\in\Int(\Z)$ such that $f(\alpha)\notin O_{\Q(\alpha)}$. This is equivalent to having $f(\alpha) \notin \oZ$, so we conclude that $\alpha$ is not in the polynomial closure of $\Z$ in $\oZ$.
\end{proof}

\begin{Conj}
For each $n \geq 1$, $\mathcal{A}_n$ is polynomially closed in $\oZ$.
\end{Conj}
\vskip0.5cm

\textbf{Acknowledgements}: The first author wishes to thank D. Dobbs for raising questions about the chain of Pr\"ufer domains $\IntQ(\mathcal{A}_n),n\in\N$ which are the original motivations of this paper. He also wishes to thank S. Mulay for the hospitality during his visit at the Department of Mathematics of the University of Tenneessee in Knoxville, USA in 2014. He thanks also S. Checcoli for stimulating discussions around Lemma \ref{bounded degree Galois closure}.

\end{document}